\documentclass[12pt,reqno]{amsart}
\usepackage[margin=1in]{geometry}
\usepackage{graphicx}
\usepackage{amsmath,amsthm,amsfonts,mathrsfs,amssymb,float,color}
\usepackage{graphicx}
\usepackage{textcomp}
\usepackage{verbatim}
\usepackage{ upgreek }
\usepackage[T1]{fontenc}
\usepackage[utf8]{inputenc}
\usepackage{hyperref}
\hypersetup{colorlinks=true,citecolor=blue,linkcolor=blue}
\usepackage{epsfig,subfigure,fancybox,balance}
\usepackage{diagbox}
\usepackage{youngtab}
\usepackage{algorithm}
\usepackage{algpseudocode}

\newcommand{\beq}[1]{\begin{equation} \label{#1}}
\newcommand{\eeq}{\end{equation}}
\newcommand{\bea}{\bed\begin{array}{rl}}
\newcommand{\eea}{\end{array}\eed}
\newcommand{\bed}{\begin{displaymath}}
\newcommand{\eed}{\end{displaymath}}
\newcommand{\barray}{\begin{array}{ll}}
\newcommand{\earray}{\end{array}}
\newcommand{\disp}{\displaystyle}
\newcommand{\ad}{&\!\disp}
\newcommand{\aad}{&\disp}
\newcommand{\al}{\alpha}
\newcommand{\e}{\varepsilon}

\newcommand{\la}{\lambda}
\newcommand{\La}{\Lambda}
\newcommand{\sg}{\sigma}

\newcommand{\ga}{\gamma}
\newcommand{\Ga}{\Gamma}
\newcommand{\dl}{\delta}
\newcommand{\Dl}{\Delta}
\newcommand{\cd}{(\cdot)}

\newcommand{\sqe}{\sqrt{\e}}

\def\phi{\varphi}

\def\half{\frac{1}{2}}

\newcommand{\CF}{{\mathcal F}}

\newcommand{\CU}{\mathcal{U}}
\newcommand{\CX}{\mathcal{X}}

\newcommand{\CO}{\mathcal{O}}

\newcommand{\CL}{\mathcal{L}}

\newcommand{\CD}{\mathcal{D}}

\newcommand{\CT}{\mathcal{T}}

\newcommand{\CY}{\mathcal{Y}}

\newcommand{\EE}{{\mathbb E}}
\newcommand{\PP}{{\mathbb P}}

\newcommand{\NN}{{\mathbb N}}
\newcommand{\rr}{{\mathbb R}}
\newcommand{\QQ}{{\mathbb Q}}

\newcommand{\lbar}{\overline}
\newcommand{\wdt}{\widetilde}
\newcommand{\wdh}{\widehat}

\newcommand{\qv}[1]{\langle #1 \rangle}
\newcommand{\bqv}[1]{\big\langle #1 \big\rangle}
\newcommand{\Bqv}[1]{\Big\langle #1 \Big\rangle}

\newcommand{\vsg}{\varsigma}
\numberwithin{equation}{section}

\newtheorem{thm}{Theorem}[section]
\newtheorem{lem}[thm]{Lemma}

\newtheorem{prop}[thm]{Proposition}

\newtheorem{exm}[thm]{Example}
\newtheorem{ass}[thm]{Assumption}

\newcommand{\thmref}[1]{Theorem~{\rm \ref{#1}}}
\newcommand{\lemref}[1]{Lemma~{\rm \ref{#1}}}

\newcommand{\figref}[1]{Figure~{\rm \ref{#1}}}
\newcommand{\assref}[1]{Assumption~{\rm \ref{#1}}}
\newcommand{\secref}[1]{Section~{\rm \ref{#1}}}

\begin{document}
\title[Partially observed optimal control of SPDEs: SMPs and numerics]{Optimal Control
of Stochastic Partial Differential Equations with Partial Observations: Stochastic Maximum Principles and Numerical Approximation
}

\author{Yanzhao Cao} 
\address{Department of Mathematics and Statistics, Auburn University, Auburn, AL, 36849}
\email{yzc0009@auburn.edu}

\author{Hongjiang Qian}
\address{Department of Mathematics and Statistics, Auburn University, Auburn, AL 36849}
\email{hjqian.math@gmail.com}

\author{George Yin}
\address{Department of Mathematics, University of Connecticut, CT 06269}
\email{gyin@uconn.edu}

\thanks{The research of H. Qian and Y. Cao was supported by the U.S. Department of Energy under the grant DE-SC0022253, DE-SC0025649 and the research of G. Yin was supported in part by the National Science Foundation under grant DMS-2204240.}


\begin{abstract}
This work
establishes a general stochastic maximum principle for partially observed optimal control of semi-linear stochastic partial differential equations in 
a nonconvex control domain. The state evolves in a Hilbert space driven by a cylindrical Wiener process and finitely many Brownian motions, while observations are in
an Euclidean space having correlated noise.
For convex control domain and 
diffusion coefficients in the state being control-independent,
numerical algorithms are developed to solve
the partially observed optimal control problems
using stochastic gradient descent algorithm combined with finite element approximations and the branching filtering algorithm.
Numerical experiments are conducted for demonstration.
\end{abstract}
\maketitle

\section{Introduction}
This work focuses on
an optimal control problem with partial observations, a topic that has been studied for several decades with numerous applications in engineering, economics, physics, and finance; see Fleming \cite{Fle68}, Bensoussan \cite{Ben92}, Pardoux \cite{Par82}, Fleming and Pardoux \cite{FP82} and references therein.
In
control theory, a major milestone was achieved in the 1950s when Pontryagin \cite{PBGM62} established a necessary condition for the optimal control of finite-dimensional deterministic systems. Since then, extensive research has focused on generalizing this result to stochastic systems. The first stochastic maximum principle was derived in \cite{Ku65}, where Kushner confirmed the existence of stochastic maximum principle and derived an
adjoint equation of the same form as that of  Pontryagin's
paper for systems with additive noise. Fast forward, the work
\cite{Pen90} extended the theory to stochastic differential equations (SDEs) with nonconvex control domains and control-dependent diffusion terms, where
Peng
characterized
the second-order adjoint process by backward stochastic differential equations (BSDEs).
Subsequently, the corresponding controlled stochastic partial differential equations (SPDEs)
were considered; the adjoint states have been characterized either by operator-valued BSDEs \cite{FZ20,LZ15,LZ18} or by stochastic bilinear forms \cite{DM13,FHT12,FHT13}.

In the context of optimal control with partial observations, stochastic maximum principles (SMPs) have been extensively studied in finite-dimensional spaces. Main
contributions include
Fleming \cite{Fle68}, Bensoussan \cite{Ben83}, Haussmann \cite{Hau87}, Baras et al. \cite{BEK89}, Li and Tang \cite{LT95}, and Tang \cite{Tan98}. In particular, Li and Tang \cite{LT95} and Tang \cite{Tan98} established a general SMP for partially observed control problems in the nonconvex control domain without using the robust form of the Zakai equation and stochastic flow method in \cite{Fle68,Hau87,Ben83,BEK89}. They considered the state process and observations being
correlated.
In contrast,
partially observed optimal control problems for SPDEs, the literature is relatively scarce.
We refer to the work of Bensoussan and Viot \cite{BV75}, Ahmed \cite{Ahm96,Ahm14,Ahm19} and the references therein. In particular, Ahmed \cite{Ahm96} proved the existence of a partially observed optimal control for stochastic evolution equations driven by additive cylindrical Wiener noise and proved the SMP using the robust form of the Zakai equation. 
{For nonlinear filtering in infinite dimensional space; we refer to \cite{AFZ97} and references therein.}

This paper continues our recent work \cite{BCQ25} on partially observed optimal control for semilinear SPDEs in a convex control domain, where the state evolves in a Hilbert space and the observation is in the Euclidean space.
In contrast to \cite{BCQ25,Ahm96}, the first contribution
of the present work is that we
establish
a general SMP for a system where the state and the observation processes have correlated
noise, and the control
appears in
all coefficients of the state-observation system, under a nonconvex control domain.
This setting introduces substantial analytically challenges that go beyond those addressed in the convex case. To overcome these challenges,
we adopt the approach in \cite{LT95,Tan98} to bypass the robust form of the Zakai equation used in \cite{Ahm96}.
%
From there, the nonconvexity of the control domain requires us to study the second-order adjoint state. Motivated by the work of Stannat and Wessels \cite{SW21}, we characterize the second-order adjoint state as the solution of a function-valued backward SPDE;
which was inspired by the related
characterization
developed in \cite{SW21} for classical (fully observed) control problems for semilinear SPDE with nonconvex control domain. It addressed certain limitations in previous work \cite{FZ20,LZ15,LZ18,DM13,FHT12,FHT13},
removing
restrictive assumptions that preclude
general Nemytskii operators or quadratic cost functions. Our
contribution pushes this further by treating partially observed system with correlated noise and fully coupled control, filling a gap in the exiting literature for the partially observed
optimal
control problem in infinite dimensional space.


The setting of our problem naturally arises in many real-world applications. In large-scale distributed systems, the full state of the system is often either physically inaccessible or of such
high dimensionality that direct monitoring becomes prohibitively and/or costly. However, partial observations are typically available through another dynamic process (such as a monitoring or observer system) whose state is fully accessible. Based on this available
possibly incomplete and imperfect, information, the objective is to optimally control the main system to achieve desired outcomes.

To further illustrate the motivation for studying the partially observed control problem in our setting, consider an aquatic ecosystem as an example. Aquatic environments such as the Great Lakes, are complex systems in which water quality and marine life depend on a delicate balance of organic and inorganic agents. The concentration of these agents, such as pollutants and nutrients, within the water body can be modeled by a partial differential equation (PDE) subject to noise, as follows
\beq{aqu}\left\{\barray\ad
\frac{\partial C}{\partial t}-D \Dl C + (\nabla C)v = b(C,u)+N_C, \quad t\geq 0, \; \xi\in \CO \\
\ad C|_{\partial \CO} =0, \quad C(0,\xi)= C_0(\xi),\, \xi \in \CO,
\earray\right.\eeq
where $\CO$ is assumed to be an open, connected, bounded domain representing the aquatic body. Here $C$ is the concentration level of $m_1$ different organic and inorganic agents such as pollutants and nutrients. The third term on the left of the first equation \eqref{aqu} represents the transport of $C$ due to water movement, where $v$ is the given velocity as a function of space-time. The function $b$ represents the interactions between $m_2$ different control agents $u$ and $m_1$ different pollutants and nutrients $C$. The $N_C$ is the distributed noise that represents the additive effect such as land runoffs from surrounding farmlands, acid rain, accidental oil spills, summer cottages, etc.

The aquatic system supports diverse species, including micro-organism and fish populations, where the stock of fish is regulated by the Department of Fisheries. The biomass per unit volume of $m_3$ different species of population $y$ can be modeled by
\beq{ob-h}
dy= h(C,y)dt+ N_y,\quad y(0)=0, t\geq 0
\eeq
with
$h$ being
the growth vector and $N_y$ being
the noise vector. The
$h$ can be
the standard logistic growth function. The Department of Fisheries and Environments is interested in designing a control policy to promote marine life and water quality, such as the application of anti-pollutants, biological agents predating unwanted microorganisms, physical removal of solid water, algae, etc. They would like to minimize their cost to achieve their goals. We refer to \cite[pp. 1593-1597]{Ahm96} for the formulation of \eqref{aqu} as equation \eqref{spde} and the corresponding partially observed optimal controlled problem. Another example of electromagnetic interference can also be found in \cite{Ahm96}.


In addition to the
contribution of the maximum principle, in this paper,
we
develop
an efficient numerical algorithm to compute the partially observed optimal control. In the classical optimal control problem, there are two main frameworks to approximate optimal controls: (i) the method based on stochastic maximum principles, and (ii) the method based on the dynamical programming principles (DPPs). The well-established DPP-based methodology is the Markov chain approximation method; see \cite{KD01,Kus77,QWY22,YWQN21} for classical optimal control, and \cite{LTT24} for partially observed optimal control in finite dimensions. It is well known that this method suffers the curse of dimensionality resulting from solving high-dimensional HJB equations.
Note that there is a recent work using deep learning to solve nonlinear filtering problem \cite{QYZ23}, which is shown to be efficient; 
{see also classical splitting-up method in \cite{ZZCZC22}}. In this paper, however, we use an alternative, namely, a
SMP-based approach for the numerical solution.
 This method is seen 
 to be efficient,
 which
 computes
 the gradient from the stochastic maximum principle and utilizes the stochastic gradient descent algorithm to iteratively update the control; see \cite{GLTZZ17} for computing the classical optimal control. For the partially observed optimal control problem, nonlinear filtering plays a significant role. One needs to combine nonlinear filtering algorithms  to calculate the conditional distribution of the state given observations. We refer to \cite{ABYZ20,WWX24} for results in finite-dimensional spaces and to \cite{BCQ25} for results in infinite-dimensional spaces, together with the references therein for further details in this direction.

Let $(\Omega,\CF,\{\CF_t\}_{t\in [0,T]},\PP)$ be a filtered probability space satisfying the usual condition, on which are defined a cylindrical Wiener process $W$ in a Hilbert space $\Xi$ and a $d$-dimensional Brownian motions $Y$ in $\rr^d$, which are assumed to be independent of each other. Consider the following stochastic partial differential equations
\beq{spde}\left\{\barray
\!\! dx_t^u\ad\!\!\!\!= [\Dl x_t^u + b(x_t^u, u_t)]dt+ \sg(x_t^u,u_t)dW_t + \sum_{j=1}^d g^j(x_t^u, u_t) dB_t^j \;\text{ on } [0,T] \times L^2(\La),\\
\!\! x_0^u\ad\!\!\!\! = x_0, \quad \text{in } L^2(\La).
\earray\right.\eeq
Here $\La\in \rr$ is a bounded interval, $b$, $\sg$, and $g^j$ are Nemytskii operators. The control $u$ takes values in the domain $U$, which is a nonempty subset of a separable Banach space $\CU$. Furthermore, $W$ is a cylindrical Wiener process in a Hilbert space $\Xi$. In what follows,
we use the Einstein summation convention by suppressing
the summation symbol
(e.g., $\sum_{j=1}^d$) in \eqref{spde}.

Suppose that system state $x_t^u$ is not completely observable, instead, one observes a function of $x_t^u$ subject to some observation noise, giving rise to the following observation process:
\beq{ob}
d Y_t = h(x_t^u, u_t) dt + dB_t,
\eeq
where $h$ is a mapping from $L^2(\La) \times U$ to $\rr^d$.
We remark that unlike classical partially observable system where the observation noise is modeled by a Wiener process $B$, we consider $Y$ {\it a priori} as a Wiener process under original measure $\PP$ on probability space $(\Omega,\CF, \CF_t, \PP)$. Imagine if we consider $B$ being
the Wiener process in classical setting, applying the Girsanov theorem implies that $Y$ will be the Wiener process under a new measure. Here we define $Y$ as a Wiener process
at the beginning. This setting is standard in the literature and is appropriate from a modeling perspective that facilitates the
derivation of the SMP; see \cite[Remark 1]{WWX24}.


Putting \eqref{ob} in \eqref{spde}, we have
\beq{new-spde}\left\{\barray
dx_t^u \ad = [\Dl x_t^u + (b-g^j h^j )(x_t^u, u_t)]dt+ \sg(x_t^u, u_t)dW_t + g^j(x_t^u, u_t)dY_t^j \\
x_0^u \ad = x_0 \in L^2(\La).
\earray\right.\eeq
To deal with the partially observed optimal control nonlinear system, we use the measure transformation technique. Let us introduce
\bea\ad
\rho_t^u:=\exp\Big\{ \int_0^t h^{j,*}(x^u_s,u_s)dY_s^j - \half \int_0^t |h(x^u_s,u_s)|^2 ds \Big\},
\eea
where $h$ is a bounded continuous function from $L^2(\La) \to \rr^d$. Then it can be shown that $z^u$ satisfies the following stochastic differential equation
\beq{rho-eq}
\left\{\barray
d\rho_t^u \ad = \rho_t^u h^*(x_t^u,u_t)dY_t, \\
\rho_0^u \ad = 1.
\earray\right.
\eeq

Define a new measure $\QQ$ by
\bea\ad
\frac{d\QQ}{d\PP}\Big|_{\CF_T}:= \rho_T^u.
\eea
Since $\rho_t^u$ is an exponential
martingale, we have $d\QQ=\rho_t^u d\PP$ on $\CF_t$. Then the Girsanov theorem implies that $W$ and $B$
are independent cylindrical Wiener processes and finite-dimensional Brownian motions, respectively, in the new probability space $(\Omega,\CF,\CF_t,\QQ)$.

The cost functional we consider is given by
\beq{cost}\barray
J(u\cd)\ad =\EE^\QQ \Big[\int_0^T \int_\La \ell(x_t^u(\la), u_t)d\la dt + \int_\La m(x^u_T(\la))d\la  \Big] \\
\aad = \EE^\PP\Big[ \int_0^T  \int_\La \rho_t^u\, \ell(x_t^u(\la),u_t)d \la dt +  \int_\La \rho_T^u\, m(x_T^u(\la)) d\la \Big],
\earray\eeq
where $\EE^\QQ$ denotes the expectation with respect to probability space $(\Omega,\CF,\CF_t,\QQ)$ and $\EE:= \EE^\PP$ denotes the expectation with respect to probability space $(\Omega, \CF, \CF_t, \PP)$. 
{Define $\CF_t^Y:=\sg(Y_s; s\leq t)$.} Our partially observed control problem is to find $\wdh u$ to minimize the cost functional \eqref{cost} over $u \in \CU_{\text{ad}}$, that is,
\beq{contr-prob}
J(\wdh u)=
\min_{u\in \CU_{\text{ad}}} J(u\cd),
\eeq
where $\CU_{\text{ad}}$ is the admissible control set defined as
\bea\ad\!\!\!\!
\CU_{\text{ad}}:=\{u: [0,T] \times \Omega \to U: u \text{ is } \CF_t^Y\text{-adapted and } \sup_{t\in [0,T]} \EE[\|u_t\|_{\CU}^k] <\infty, \forall\, k\in \NN \}.
\eea

Let $\wdh u$ be an optimal control of the partially observed problem \eqref{contr-prob}, and let $\wdh x$ be the associated optimal state. The goal of this paper is to establish 
necessary conditions for the partially observed optimal control $\wdh u$ and to develop efficient numerical algorithms to compute $\wdh u$. To the best of
our knowledge, this is the first study on a general SMP for partially observed optimal control of SPDEs with correlated state-observation noise, where the control affects all diffusion coefficients in a nonconvex domain.
In addition, we also develop efficient numerical algorithms.

The rest of the
paper is organized as follows. In \secref{sec:ass}, we present the notation, assumptions, and main results, namely, Pontryagin's maximum principle and Peng's maximum principle under convex and nonconvex control domain, respectively. In \secref{sec:var}, we study the variational inequalities of the spike variation to prove Peng's maximum principle. In \secref{sec:adjoint}, we characterize the adjoint equation by backward SPDEs. \secref{sec:algo} provides numerical algorithms to solve the partially observed optimal control of SPDEs. Finally,
an numerical example is presented in \secref{sec:num} for demonstration.

\section{Notation and assumptions}\label{sec:ass}

Let $\La \subset \rr$ be a bounded interval. For $\ga>0$, let $H_0^\ga (\La):= W_0^{\ga,2}$ be the fractional Sobolev space of order $\ga$ with Dirichlet boundary conditions and $H^{-\ga}$ as its dual space. Let $T$ be a finite time horizon, which is fixed throughout the paper. We consider the state equation \eqref{spde} in the variational setting on the Gelfand triple:
\bea
H_0^1(\La) \hookrightarrow L^2(\La) \hookrightarrow H^{-1}(\La).
\eea
Denote by $\Dl: H_0^1(\La) \to H^{-1}(\La)$ as a realization of continuous operators.
For real separable Hilbert spaces $\CX,\CY$, we denote by $\CL_2(\CX,\CY)$ the space of Hilbert-Schmidt operators from $\CX$ to $\CY$, and by $\CL_1(\CX,\CY)$ the space of trace class operators from $\CX$ to $\CY$. If $\CX=\CY$, we denote $\CL_2(\CX):=\CL_2(\CX,\CX)$ and $\CL_1(\CX):=\CL_1(\CX,\CX)$. For an operator $\CT$, $\CT^*$ denotes its adjoint.
{For $x=(x^1,\dots,x^d)$ with $x^j \in L^2(\La)$, we will write $x\in (L^2(\La))^d$.}
We use $K$ as a generic constant that can vary from place to place.

To proceed, we make the following assumptions.

\begin{ass}\label{ass}
We assume
\begin{itemize}
    \item[{\rm(H1)}] The function $b:\rr\times U \to \rr$, $\sg: \rr \times U \to \CL_2(\Xi,\rr)$, $g: \rr \times U\to \CL(\rr^d,\rr)$, $h:\rr \times U \to \rr^d$, $\ell: \rr \times U \to \rr$, and $m:\rr \to \rr$ are continuous in $u$ and twice continuously differentiable in $x$. We assume $b_x,b_{xx},\sg_x, \sg_{xx}, g_{x}, g_{xx}, \ell_{xx}, m_{xx}$ are bounded, and $b,\sg, g, \ell_x, m_x$ are bounded by $K(1+|x|+\|u\|_{\CU}), x\in \rr, u\in U$.
    \item[{\rm(H2)}] The mapping $h: L^2(\La) \times U \to \rr^d$ is bounded with bounded  Gateaux derivative.
    \item[{\rm(H3)}] $x_0 \in L^2(\La)$.
\end{itemize}
\end{ass}
We define the Nemytskii operators in $L^2(\La)$ for the above coefficients. Taking $\sg$ as an example, we can define the Nemytskii operator
\bea
[\sg(x,u)y](\la) := \sg(x(\la), u)y, \quad \la \in \La,\; y \in \Xi.
\eea
Define the Hamiltonian $H:L^2(\La)\times U \times L^2(\La) \times \CL_2(\Xi; L^2(\La))\times (L^2(\La))^d \times \rr^d \to \rr$:
\beq{def-H}\barray
H(x,u,p,q_1,q_2 ,z_2) \ad :=\qv{p, b(x,u)}_{L^2(\La)} +\qv{q_1, \sg(x,u)}_{\CL_2(\Xi,L^2(\La))} \\
\aad\quad + \qv{q_2^j, g^j(x,u)}_{L^2(\La)} + L(x,u) + z_2^j h^j(x,u),
\earray\eeq
where $L: L^2(\La) \times U \to \rr$ is defined as $L(x,u):=\int_\La \ell(x(\la), u) d\la$. (Recall that we are using the  Einstein summation convention.)

Consider the following backward stochastic differential equations (BSDEs)
\beq{z}\left\{\barray
dz_t = -\ell(\wdh x_t, \wdh u_t) dt + z_{1,t} dW_t + z_{2,t}^j dB_t^j \\
z_T = m(\wdh x_T),
\earray\right.
\eeq
and backward stochastic partial differential equations (BSPDEs)
\beq{pq}\left\{\barray
dp_t  \ad = -[\Dl p_t +b_x(\wdh x_t, \wdh u_t) p_t+ \qv{\sg_x(\wdh x_t, \wdh u_t), q_{1,t}}_{\CL_2(\Xi;L^2(\La))}+ g_x^j(\wdh x_t, \wdh u_t)q_{2,t}^j \\
\aad \qquad + h_x^{j,*}(\wdh x_t, \wdh u_t)z_{2,t}^j - g^j(\wdh x_t, \wdh u_t)h_x^j(\wdh x_t, \wdh u_t) p_t
+ \ell_x(\wdh x_t, \wdh u_t)] dt \\
\aad \qquad + q_{1,t} dW_t + q_{2,t}^{j} dB_t^j
\\
p_T \ad = m_x(\wdh x_T).
\earray\right.\eeq
Under \assref{ass}, they can be shown to have a unique $\CF_t$-adapted strong solution $z, z_{2}^j \in L^2([0,T]\times \Omega; \rr)$ and $z_1\in L^2([0,T]\times \Omega; \CL_2(\Xi;\rr))$ to \eqref{z} (see \cite{LT95}) and a unique variational solution $(p,q_1,q_2), q_2 =(q_2^1,\dots, q_2^d)$ to \eqref{pq}, where
\bea
p\in L^2([0,T]\times \Omega; H_0^1(\La)) \cap L^2(\Omega; C([0,T];L^2(\La)))
\eea
and $
q_1 \in L^2([0,T]\times \Omega; \CL_2(\Xi; L^2(\La))),\; q_2^j \in L^2([0,T]\times \Omega;L^2(\La)),\; j=1,2,\dots, d$; see \cite{Ben83,SW21}.

\subsection{Main results}

\begin{thm}[Pontryagin's maximum principle]
\label{thm:pon-smp}
Let $(\wdh x, \wdh u)$ be an optimal pair of the partially observed optimal control problem \eqref{new-spde}--\eqref{contr-prob}. Suppose that the control domain is convex.  Then there exist $(z,z_1,z_2)$ and $(p,q_1,q_2)$ satisfying \eqref{z} and \eqref{pq}, respectively,  such that for any $v\in \CU_{\text{ad}}$ and a.e. $(t,\omega)\in [0,T]\times \Omega$,
\bea\ad
\bqv{\EE^\QQ[\nabla_u H(\wdh x_t, \wdh u_t, p_t, q_{1,t}, q_{2,t}, z_{2,t})| \CF_t^Y], v- \wdh u_t} \\
\aad = \Bqv{\EE^\QQ \Big[b_u^*(\wdh x_t, \wdh u_t) p_t + \sum_{i=1}^\infty [\sg(\wdh x_t, \wdh u_t)e_i]_u^* q_{1, t} e_i + \sum_{j=1}^d g_u^{j,*} (\wdh x_t, \wdh u_t) q_{2, t}^j \\
\aad \qquad + L_u(\wdh x_t, \wdh u_t) + \sum_{j=1}^d h_u^{j,*}(\wdh x_t,\wdh u_t) z_{2,t}^j \Big| \CF_t^Y \Big], v-\wdh u_t} \geq 0,
\eea
where $\{e_i\}$ is an orthonormal basis of $L^2(\La)$.
\end{thm}

The proof of \thmref{thm:pon-smp} is omitted and can be found in
\cite{Ben83,LT95,Tan98,FHT18}; see also our work \cite[Theorem 2.3]{BCQ25}.
The rest of paper is
devoted to proving a more general SMP,  the Peng's maximum principle, for our partially observed optimal control problem.
We will use the following identification as in \cite{SW21};
\bea
L^2(\La, \CL_2(\Xi,\rr)) \ad \approxeq \CL_2(\Xi, L^2(\La)) \\
q(\la)(\xi) \ad \leftrightarrow q(\xi)(\la).
\eea

\begin{thm}[Peng's maximum principle]\label{thm:Peng-smp}
Let $(\wdh x, \wdh u)$ be an optimal pair of the partially observed control problem \eqref{new-spde}--\eqref{contr-prob}. Suppose that the control domain is non-convex. Then there exist adapted process $(z,z_1,z_2)$, $(p,q_1,q_2)$, and $(P,Q_1,Q_2)$ such that for all $v\in \CU_{\text{ad}}$ and almost every $(t,\omega)\in [0,T]\times\Omega$,
\bea\ad
\!\!\!\! \EE^\QQ \Big[H \big(\wdh x_t, v, p_t, q_{1,t}, q_{2,t}, z_{2,t}-\qv{p_t, g(\wdh x_t, \wdh u_t)}_{L^2(\La)} \big) \\
\aad \quad - H \big(\wdh x_t, \wdh u_t, p_t, q_{1,t}, q_{2,t}, z_{2,t}-\qv{p_t, g(\wdh x_t, \wdh u_t)}_{L^2(\La)} \big) \Big| \CF_t^Y \Big] \\
\aad\!\!\!\!\!\!\!\!\!+  \EE^\QQ \Big[\half \bqv{P_t(\la,\mu), \qv{\sg(\wdh x_t(\la), v)-\sg(\wdh x_t(\la), \wdh u_t), \sg(\wdh x_t(\mu), v)- \sg(\wdh x_t(\mu), \wdh u_t)}_{\CL_2(\Xi,\rr)}} \Big| \CF_t^Y\Big] \\
\aad\!\!\!\!\!\!\!\!\! + \EE^\QQ \Big[\half \bqv{P_t(\la,\mu), \qv{g^j(\wdh x_t(\la), v)-g^j(\wdh x_t(\la), \wdh u_t), \\
\aad \qquad \qquad \qquad g^j(\wdh x_t(\mu), v)- g^j(\wdh x_t(\mu), \wdh u_t)}_{\rr} } \big| \CF_t^Y \Big]\geq 0,
\eea
where $(z,z_1,z_2)$ and $(p,q_1,q_2)$ are solutions of \eqref{z} and \eqref{pq}, and $(P,Q_1,Q_2)$ are the solution of the second order adjoint equations given by \thmref{thm:PQ}.
\end{thm}

In what follows,
we will prove \thmref{thm:Peng-smp} in \secref{sec:var} and \secref{sec:adjoint}.

\section{Variational inequality}\label{sec:var}
\subsection{Spike variation} Since the control domain is non-convex, we adopt the spike variation argument as in Peng \cite{Pen90}. Namely, fix any $v\in \CU_{\text{ad}}$, $\tau \in (0,T)$, we define
\beq{ue}
u_t^\e:= \left\{\barray v, \quad \text{ if } \tau \leq t \leq \tau+\e, \\
\wdh u_t, \quad \text{ otherwise.}
\earray\right.
\eeq
Let $x_t^\e, \rho_t^\e$ be the solutions of \eqref{new-spde} and \eqref{rho-eq} corresponding to $u^\e$.

\subsection{Derivation of 
variational inequality} We look for the cost variation of $J(u^\e)-J(\wdh u)$. It involves variations of $x^\e - \wdh x$ and $\rho^\e-\wdh \rho$. To proceed, we introduce the following abbreviations:
$\dl f(t;u^\e_t):= f(\wdh x_t, u^\e_t) - f(\wdh x_t,\wdh u_t)$, and $\dl f(t,\la; u^\e_t):= f(\wdh x_t(\la), u_t^\e) - f(\wdh x_t(\la), \wdh u_t)$ for $f=b, g^j, h^j, (b-g^j h^j ), \ell, H$ as well as their derivatives in $x$. We will use $\dl f(t;u_t^\e)$ (resp. $\dl f(t,\la; u_t^\e)$) and $f(\wdh x_t, u_t^\e)-f(\wdh x_t, \wdh u_t)$ (resp. $f(\wdh x_t(\la), u_t^\e)- f(\wdh x_t(\la), \wdh u_t)$) interchangeably for the rest of the paper.

Let $x^{1,\e}$ be the solution of the SPDE
\beq{x1}
\left\{\barray
dx_t^{1,\e} \ad = [\Dl x_t^{1,\e} + (b-g^jh^j)_x(\wdh x_t, \wdh u_t) x_t^{1,\e}]dt \\
\aad\quad + [\sg_x(\wdh x_t, \wdh u_t) x_t^{1,\e} + \sg(\wdh x_t,u_t^\e)-\sg(\wdh x_t, \wdh u_t)]dW_t \\
\aad \quad + [g_x^j(\wdh x_t, \wdh u_t) x_t^{1,\e} + g^j(\wdh x_t, u_t^\e)-g^j(\wdh x_t, \wdh u_t)]dY_t^j, \\
x_0^{1,\e} \ad = 0,
\earray\right.
\eeq
and $\rho_t^{1,\e}$ be the solution of SDE:
\beq{rho1}\left\{\barray
d\rho_t^{1,\e} \ad=  \{ \rho^{1,\e}_t h^j(\wdh x_t, \wdh u_t)+\wdh\rho_t [h_x^j(\wdh x_t, \wdh u_t) x_t^{1,\e} + (h^j(\wdh x_t, u_t^\e)- h^j(\wdh x_t, \wdh u_t))]\} dY_t^j, \\
\rho_0^{1,\e}\ad =0. \earray\right.\eeq

Equations \eqref{x1} and \eqref{rho1} are the first variational equations of $x^\e-\wdh x$ and $\rho^\e-\wdh \rho$, respectively. The second variational equations are given by the solution of $(x_t^{2,\e}, \rho_t^{2,\e})$ satisfying:
\beq{x2}\left\{\barray
\!\!\! dx_t^{2,\e} \ad \!\!\!\!\!= [\Dl x_t^{2,\e} + (b-g^jh^j)_x(\wdh x_t, \wdh u_t) x_t^{2,\e}]dt \\\aad\!\!\!\!\!+ [(b-g^j h^j)(\wdh x_t, u_t^\e)- (b-g^j h^j)(\wdh x_t, \wdh u_t)+\half (b-g^j h^j)_{xx}(\wdh x_t, \wdh u_t) x_t^{1,\e}x_t^{1,\e}]dt  \\
\aad\!\!\!\!\! + \big[ \sg_x(\wdh x_t, \wdh u_t) x_t^{2,\e}+ (\sg_x(\wdh x_t, u_t^\e)-\sg_x(\wdh x_t, \wdh u_t)) x_t^{1,\e} + \half \sg_{xx}(\wdh x_t, \wdh u_t) x_t^{1,\e} x_t^{1,\e} \big]dW_t \\
\aad\!\!\!\!\! + \big[g_x^j(\wdh x_t, \wdh u_t)x_t^{2,\e}+ (g_x^j(\wdh x_t, u_t^\e)-g_x^j(\wdh x_t, \wdh u_t)) x_t^{1,\e}+\half g_{xx}^j(\wdh x_t, \wdh u_t) x_t^{1,\e}x_t^{1,\e}\big] dY_t^j, \\
x_0^{2,\e}\ad\!\!\!\!\! =0,
\earray\right.
\eeq
and
\beq{rho2}\left\{\barray
d\rho_t^{2,\e} \ad = [\rho_t^{2,\e} h^j(\wdh x_t, \wdh u_t)+ \wdh \rho_t h_x^j(\wdh x_t, \wdh u_t) x_t^{2,\e}+ (h^j(\wdh x_t, u_t^\e)- h^j(\wdh x_t, \wdh u_t))\rho_t^{1,\e} \\
\aad \quad + \wdh \rho_t (h_x^j(\wdh x_t, u_t^\e)-h_x^j(\wdh x_t, \wdh u_t)) x_t^{1,\e}+ \half \wdh \rho_t h_{xx}^j(\wdh x_t, \wdh u_t) x_t^{1,\e} x_t^{1,\e}\\
\aad\quad + \rho_t^{1,\e} h_x^j(\wdh x_t, u_t^\e)x_t^{1,\e}\big] d Y_t^j,  \\
\rho_0^{2,\e}\ad =0.
\earray\right.
\eeq

Compared to classical optimal control without partial observations, we have two pairs $(x^{1,\e},x^{2,\e})$ and $(\rho^{1,\e}, \rho^{2,\e})$ of variational equations. It is because our state process have two components $(x,\rho)$ after using the Girsanov transformation. Following \cite{SW21} and \cite{LT95}, one can establish the following estimates
\begin{lem}\label{var-est}
We have
\bea\disp
\sup_{t\in [0,T]} \EE^\QQ[\|x_t^{1,\e}\|_{L^2(\La)}^{2k} ] \ad \leq K \e^k, \quad  \sup_{t\in [0,T]} \EE^\QQ \big [\|x_t^{2,\e}\|_{L^2(\La)}^k \big] \leq K \e^k, \\
\disp \sup_{t\in [0,T]}\EE^\QQ [|\rho_t^{1,\e}|^{2k}] \ad \leq K \e^k, \quad \sup_{t\in [0,T]} \EE^\QQ [|\rho_t^{2,\e}|^k] \leq K \e^k,
\eea
and
\bea\disp
\sup_{t\in [0,T]}\EE^\QQ \big[\|x_t^\e- \wdh x_t - x_t^{1,\e}- x_t^{2,\e}\|_{L^2(\La)}^2 \big] \ad \leq o(\e^2), \\
\disp \sup_{t\in [0,T]} \EE^\QQ \big[|\rho_t^\e -\wdh \rho_t - \rho_t^{1,\e} - \rho_t^{2,\e}|^2 \big] \ad \leq o(\e^2).
\eea
\end{lem}
\begin{proof}
The proof of the above inequalities can be found in \cite[Lemma 3.1, 3.2]{SW21} and \cite[Lemma 3.2]{LT95}. Thus, details are omitted.
\end{proof}

We are now in a position to establish the variational inequality for the cost functional.
\begin{lem}\label{lem:var-cost}
We have
\beq{var-cost-P}\barray\ad
\EE^\PP \Big\{ \int_0^T \int_{\La} \wdh\rho_t \big[\ell_x(\wdh x_t(\la), \wdh u_t)(x_t^{1,\e}(\la)+ x_t^{2,\e}(\la)) +\half \ell_{xx}(\wdh x_t(\la), \wdh u_t
) x_{t}^{1,\e}(\la) x_t^{1,\e}(\la) \\
\aad \qquad \qquad\qquad + \ell(\wdh x_t(\la), u_t^\e)-\ell(\wdh x_t(\la), \wdh u_t) \big] d\la dt \\
\aad \qquad\quad + \int_{\La} \wdh\rho_T \Big[ m_x(\wdh x_T(\la))(x_T^{1,\e}(\la) + x_T^{2,\e}(\la)) +\half m_{xx}(\wdh x_T(\la))x_T^{1,\e}(\la) x_T^{1,\e}(\la) \Big]  d\la \Big\} \\
\aad + \EE^\PP \Big\{
\int_0^T \int_{\La} \Big[(\rho_t^{1,\e} + \rho_t^{2,\e}) \ell(\wdh x_t(\la), \wdh u_t) + \rho_t^{1,\e} \ell_x(\wdh x_t(\la), \wdh u_t) x_t^{1,\e}(\la ) \Big] d\la dt \\
\aad\qquad\quad  + \int_{\La} \big[ (\rho_T^{1,\e}+\rho_T^{2,\e}) m(\wdh x_T(\la)) + \rho_T^{1,\e} m_x(\wdh x_T(\la)) x_T^{1,\e}(\la) \big] d\la
\Big\} \leq o(\e).
\earray\eeq
Equivalently,
\beq{var-cost2}\barray\ad
\EE^\QQ \Big\{ \int_0^T \int_{\La} \big[\ell_x(\wdh x_t(\la), \wdh u_t)(x_t^{1,\e}(\la)+ x_t^{2,\e}(\la)) +\half \ell_{xx}(\wdh x_t(\la), \wdh u_t
) x_{t}^{1,\e}(\la) x_t^{1,\e}(\la) \\
\aad \qquad \qquad\qquad + \ell(\wdh x_t(\la), u_t^\e)-\ell(\wdh x_t(\la), \wdh u_t) \big] d\la dt \\
\aad \qquad\quad + \int_{\La}  \Big[ m_x(\wdh x_T(\la))(x_T^{1,\e}(\la) + x_T^{2,\e}(\la)) +\half m_{xx}(\wdh x_T(\la))x_T^{1,\e}(\la) x_T^{1,\e}(\la) \Big]  d\la \Big\} \\
\aad + \EE^\QQ \Big\{
\int_0^T \int_{\La} \wdh\rho^{-1}_t \big[(\rho_t^{1,\e} + \rho_t^{2,\e}) \ell(\wdh x_t(\la), \wdh u_t) + \rho_t^{1,\e} \ell_x(\wdh x_t(\la), \wdh u_t) x_t^{1,\e}(\la ) \big] d\la dt \\
\aad\qquad\quad  + \int_{\La} \wdh\rho^{-1}_T\big[ (\rho_T^{1,\e}+\rho_T^{2,\e}) m(\wdh x_T(\la)) + \rho_T^{1,\e} m_x(\wdh x_T(\la)) x_T^{1,\e}(\la) \big] d\la
\Big\} \\
\aad := \EE_1^\QQ +\EE_2^\QQ \leq o(\e).
\earray\eeq
\end{lem}
\begin{proof}
The proof of this lemma can be obtained from \cite[Lemma 3.4]{SW21} and \cite[Lemma 3.3]{LT95}. The details are omitted.
\end{proof}

\section{Adjoint representation of variational inequality}\label{sec:adjoint}

\subsection{Rewrite of $\EE_1^\QQ+\EE_2^\QQ$} The It\^{o} formula implies that the solution of \eqref{rho1} and \eqref{rho2} have the following explicit expressions:
\beq{rho1-int}
\rho_t^{1,\e}= \wdh \rho_t
\int_0^t [h_x^j(\wdh x_s, \wdh u_s) x_s^{1,\e} + (h^j(\wdh x_s, u_s^\e)- h^j(\wdh x_s, \wdh u_s))]dB_s^j,
\eeq
and
\beq{rho2-int}\barray
\rho_t^{2,\e} \ad = \wdh \rho_t \int_0^t [h_x^j(\wdh x_s, \wdh u_s) x_s^{2,\e} + \dl h_x^j(s;u_s^\e) x_s^{1,\e} + \half h_{xx}^j(\wdh x_s, \wdh u_s) x_s^{1,\e} x_s^{1,\e}] \\
\aad\quad +  \rho_s^{-1} \rho_s^{1,\e}[h^j(\wdh x_s, u_s^\e)- h^j(\wdh x_s, \wdh u_s)+ h_x^j(\wdh x_s, \wdh u_s)x_s^{1,\e}] dB_s^j.
\earray\eeq
Set
\beq{def-Ga}
\Ga_t:= (\wdh \rho_t)^{-1} (\rho_t^{1,\e}+ \rho_t^{2,\e}).
\eeq
Then \eqref{rho1-int} and \eqref{rho2-int} implies that
\bea\ad
\Ga_t= \int_0^t \ga_s^j dB_s^j,
\eea
where $\ga_s:=(\ga_s^1,\dots, \ga_s^d)$ with $\ga_s^j$ defined as
\beq{ga}\barray
\ga_s^j \ad := h_x^j(\wdh x_s,\wdh u_s)(x_s^{1,\e}+x_s^{2,\e})+ [h^j(\wdh x_s, u_s^\e) - h^j(\wdh x_s, \wdh u_s)] \\
\aad\quad  + \dl h_x^j(s; u_s^\e) x_s^{1,\e} + \half h_{xx}^j(\wdh x_s, \wdh u_s) x_s^{1,\e} x_s^{1,\e} \\
\aad\quad + \wdh \rho_s^{-1} \rho_s^{1,\e} [h^j(\wdh x_s, u_s^\e) - h^j(\wdh x_s, \wdh x_s)+ h_x^j(\wdh x_s, \wdh u_s) x_s^{1,\e}].
\earray\eeq
We note that
\beq{Ga-0T}
\Ga_0 =0 , \quad \Ga_T =(\wdh \rho_T)^{-1}(\rho_T^{1,\e}+ \rho_T^{2,\e}).
\eeq
Applying the It\^{o} formula to $z_t \Ga_t$, we have
\beq{zGa}
\EE^\QQ z_T \Ga_T = \EE^\QQ z_0 \Ga_0 + \EE^\QQ \int_0^T \big[ \qv{z_{2,s}, \ga_s}_{\rr^d} - \Ga_s \ell(\wdh x_s, \wdh u_s) \big]ds.
\eeq
Then plugging \eqref{def-Ga}, \eqref{ga}, and \eqref{Ga-0T} into \eqref{zGa}, we obtain
\beq{E2Q-1}\barray \ad\!\!\!\!
\EE^\QQ\Big\{\int_0^T \int_\La (\wdh \rho_t)^{-1}[\rho_t^{1,\e}+\rho_t^{2,\e}] \ell(\wdh x_t(\la), \wdh u_t)d\la dt + \int_\La (\wdh \rho_T)^{-1}[\rho_T^{1,\e} + \rho_T^{2,\e}] m(\wdh x_T(\la)) d\la \Big\} \\
\aad\!\!\!\! = \EE^\QQ \int_0^T \qv{z_{2,s}^j, h_x^j(\wdh x_s, \wdh u_s)(x_s^{1,\e}+x_s^{2,\e})+ [h^j(\wdh x_s, u_s^\e)- h^j(\wdh x_s, \wdh u_s)] \\
\aad\!\!\!\! \qquad \qquad\quad  + \half h_{xx}^j(\wdh x_s, \wdh u_s) x_s^{1,\e} x_s^{1,\e}}_{\rr} ds \\
\aad\!\!\!\! \quad + \EE^\QQ \int_0^T \qv{z_{2,s}^j, h_x^j(\wdh x_s, \wdh u_s)[(\wdh\rho_s)^{-1} \rho_s^{1,\e} x_s^{1,\e}]}_{\rr} dt +o(\e).
\earray\eeq
In the above equality, from \lemref{var-est}, we used the fact that
\bea\ad
\EE^\QQ \int_0^T  \qv{z_{2,s}^j,  \dl h_x^j(s; u_s^\e) x_s^{1,\e} + (\wdh{\rho_s})^{-1} \rho_s^{1,\e}[h^j(\wdh x_s, u_s^\e)- h^j(\wdh x_s, \wdh u_s)]}_{\rr} ds = o(\e).
\eea

Consider remaining terms in \eqref{var-cost2} of $\EE_2^\QQ$, we need to study $[(\wdh \rho_s)^{-1} \rho_s^{1,\e}]x_s^{1,\e}$. Define $
X_s^{1,\e}:=[(\wdh \rho_s)^{-1} \rho_s^{1,\e}] x_s^{1,\e}$. From \eqref{rho1-int}, we have
\bea
X_t^{1,\e} \ad = x_t^{1,\e} \int_0^t [h_x^j(\wdh x_s, \wdh u_s) x_s^{1,\e} + h^j(\wdh x_s, u_s^\e)- h^j(\wdh x_s, \wdh u_s)]dB_s^j.
\eea
The It\^{o} formula for variational solutions for SPDEs in \cite[Theorem 4.2.5]{LR15} implies that the dynamics of $X_t^{1,\e}$ follows the equation:
\bea
d X_t^{1,\e}\ad = [\Dl X_t^{1,\e}+(b-g^j h^j)_x X_t^{1,\e}]dt \\
\aad +\Big\{g_x^j(\wdh x_t, \wdh u_t) h^j(\wdh x_t, \wdh u_t) X_t^{1,\e} + \wdh \rho_t^{-1} \rho_{t}^{1,\e} \dl g^j(t;u^\e_t) h^j(\wdh x_t, \wdh u_t) \Big\} dt \\
\aad + \Big\{ \big(g_x^j(\wdh x_t, \wdh u_t) x_t^{1,\e} +\dl g^j(t;u_t^\e)\big) \big(h_x^j(\wdh x_t, \wdh u_t) x_t^{1,\e} + \dl h^j(t; u_t^\e) \big) \Big\}dt  \\
\aad+ \Big\{ \sg_x(\wdh x_t, \wdh u_t) X_t^{1, \e} + \wdh \rho_t^{-1} \rho_t^{1,\e} \dl \sg(t; u_t^\e)  \Big\} dW_s \\
\aad+ \Big\{ \big(g_x^j( \wdh x_t, \wdh u_t) X_t^{1,\e}+ \wdh \rho_t^{-1} \rho_t^{1,\e} \dl g^j(t; u_t^\e) \big) \\
\aad \qquad + x_t^{1,\e} \big(h_x^j(\wdh x_t, \wdh u_t)x_t^{1,\e} +\dl h^j(t;u_t^\e)\big) \Big\} d B_s^j.
\eea

We then apply the It\^{o} formula to process $\qv{p_t,X_t^{1,\e}}_{L^2(\La)}$. It gives
\bea\ad
d \qv{p_t, X_t^{1,\e}}_{L^2(\La)}  \\
\ad = \qv{p_t, dX_t^{1,\e}}_{L^2(\La)}+ \qv{X_t^{1,\e}, dp_t}_{L^2(\La)} + d\qv{p,X^{1,\e}}_t \\
\ad = \bqv{p_t, \Dl X_t^{1,\e}+ (b-g^j h^j)_x X_t^{1,\e}}_{L^2(\La)} dt\\
\aad + \bqv{p_t, g_x^j(\wdh x_t, \wdh u_t) h^j(\wdh x_t, \wdh u_t) X_t^{1,\e}+ \wdh \rho_t^{-1}\rho_t^{1,\e} \dl g^j(t; u_t^\e) h^j(\wdh x_t, \wdh u_t)}_{L^2(\La)} dt \\
\aad + \bqv{p_t, (g_x^j(\wdh x_t, \wdh u_t) x_t^{1,\e} + \dl g^j(t;u_t^\e)) (h_x^j(\wdh x_t, \wdh u_t) x_t^{1,\e} +\dl h^j(t; u_t^\e))}_{L^2(\La)} dt \\
\aad + \bqv{p_t, (\sg_x(\wdh x_t, \wdh u_t) X_t^{1,\e}+ \wdh \rho_t^{-1} \rho_t^{1,\e} \dl \sg(t;u_t^\e) ) dW_t}_{L^2(\La)} \\
\aad  + \bqv{p_t, [g_x^j(\wdh x_t, \wdh u_t) X_t^{1,\e}+ \wdh \rho_t^{-1} \rho_t^{1,\e} \dl g^j(t;u^\e_t) \\
\aad \qquad\quad  + x_t^{1,\e} h_x^j(\wdh x_t, \wdh u_t) x_t^{1,\e} + x_t^{1,\e} \dl h^j(t; u^\e_t)] dB_t^j}_{L^2(\La)} \\
\aad + \bqv{X_t^{1,\e}, -\Dl p_t - b_x(\wdh x_t, \wdh u_t) p_t - \qv{\sg_x(\wdh x_t, \wdh u_t), q_{1,t}}_{\CL_2(\Xi; L^2(\La))} - g_x^j(\wdh x_t, \wdh u_t)q^j_{2,t} \\
\aad \qquad \qquad - h_x^{j,*}(\wdh x_t, \wdh u_t) z_{2,t}^j + g^j(\wdh x_t, \wdh u_t) h_x^j(\wdh x_t, \wdh u_t) p_t  - \ell_x(\wdh x_t, \wdh u_t)}_{L^2(\La)} dt \\
\aad + \qv{X_t^{1,\e}, q_{1,t} dW_t}_{L^2(\La)} + \qv{X_t^{1,\e}, q_{2,t} dB_t}_{L^2(\La)} \\
\aad + \qv{q_{1,t}, \sg_x(\wdh x_t, \wdh u_t) X_t^{1,\e} + \wdh \rho_t^{-1} \rho_t^{1,\e} \dl \sg(t; u^\e_t) }_{\CL_2(\Xi; L^2(\La))} dt  \\
\aad + \bqv{q_{2,t}^j, g_x^j(\wdh x_t, \wdh u_t) X_t^{1,\e} + \wdh \rho_t^{-1} \rho_t^{1,\e} \dl g^j(t; u_t^\e) \\
\aad\qquad \quad  + x_t^{1,\e} h_x^j(\wdh x_t, \wdh u_t) x_t^{1,\e} + x_t^{1,\e} \dl h^j(t; u_t^\e)}_{L^2(\La)}dt.
\eea
Taking integration and expectation, and considering the initial and terminal conditions, we obtain
\bea\ad\!\!\!\!
\EE^\QQ \qv{m_x(\wdh x_T), X_T^{1,\e}}_{L^2(\La)} \\
\ad\!\!\!\! = \EE^\QQ \qv{p_T, X_T^{1,\e}}_{L^2(\La)} \\
\ad \!\!\!\!= - \EE^\QQ\int_0^T \int_\La \big[z_{2,t}^j h_x^j(\wdh x_t(\la), \wdh u_t(\la))+\ell_x(\wdh x_t(\la), \wdh u_t(\la))\big] X_t^{1,\e}(\la) d\la dt \\
\aad \!\!\!\!+\EE^\QQ \int_0^T \qv{p_t,  \wdh \rho_t^{-1} \rho_t^{1,\e} \dl g^j (\wdh x_t, \wdh u_t) h^j(\wdh x_t, \wdh u_t)}_{L^2(\La)}dt \\
\aad\!\!\!\!+ \EE^\QQ \int_0^T \qv{p_t, (g_x^j(\wdh x_t, \wdh u_t) x_t^{1,\e}+\dl g^j(t; u^\e_t))(h_x^j(\wdh  x_t, \wdh u_t) x_t^{1,\e}+ \dl h^j(t; u^\e_t))}_{L^2(\La)}dt  \\
\aad\!\!\!\!+ \EE^\QQ \int_0^T \qv{q_{1,t},\wdh \rho_t^{-1} \rho_t^{1,\e} \dl \sg(t; u^\e_t)}_{\CL_2(\Xi; L^2(\La))} dt \\
\aad\!\!\!\!+ \EE^\QQ \int_0^T \qv{q_{2,t}^j, \wdh \rho_t^{-1} \rho_t^{1,\e} \dl g^j(t; u^\e_t) + [h_x^j(\wdh x_t, \wdh u_t) x_t^{1,\e}] x_t^{1,\e} + \dl h^j(t; u^\e_t) x_t^{1,\e} }_{L^2(\La)} dt.
\eea
From \lemref{var-est}, we further note
that
\bea\ad
\EE^\QQ  \int_0^T \qv{p_t, \wdh \rho_t^{-1}\rho_t^{1,\e} \dl g^j(t; u^\e_t) h^j(\wdh x_t, \wdh u_t)}_{L^2(\La)}dt =o(\e), \\
\ad\EE^\QQ \int_0^T \qv{p_t, g_x^j(t) x_t^{1,\e} \dl h(t;u^\e_t)+ h_x^j(t)x_t^{1,\e} \dl g^j(t;u_t^\e)\\
\aad \qquad\qquad\quad +\dl g^j(t;u_t^\e) \dl h^j(t;u^\e_t)}_{L^2(\La)} dt = o(\e), \\
\ad \EE^\QQ \int_0^T \qv{q_{1,t}, \wdh \rho_t^{-1} \rho_t^{1,\e} \dl \sg(t; u_t^\e)}_{\CL_2(\Xi; L^2(\La))}dt = o(\e), \\
\ad \EE^\QQ \int_0^T  \qv{q_{2,t}^j, \wdh \rho_t^{-1} \rho_t^{1,\e} \dl g^j(t;u_t^\e) + x_t^{1,\e} \dl h^j(t; u_t^\e)}_{L^2(\La)} dt = o(\e),
\eea
where $g_x^j(t):=g_x^j(\wdh x_t, \wdh u_t)$ and $h_x^j(t):=h_x^j(\wdh x_t, \wdh u_t)$. Consequently, we conclude that
\beq{E2Q-2}\barray\ad
\EE^\QQ \int_{\La} m_x(\wdh x_T(\la))X_T^{1,\e}(\La) d\la \\
\aad\quad + \EE^\QQ \int_0^T \int_\La (\ell_x (\wdh x_t(\la), \wdh u_t(\la))+ z_{2,t}^j h_x^j(\wdh x_t(\la), \wdh u_t(\la)) X_t^{1,\e})d\la dt \\
\aad = \EE^\QQ \int_0^T \qv{p_t, (g_x^j(\wdh x_t, \wdh u_t) x_t^{1,\e})(h_x^j(\wdh x_t, \wdh u_t) x_t^{1,\e})}_{L^2(\La)} dt \\
\aad \quad + \EE^\QQ \int_0^T \qv{q_{2,t}^j,  x_t^{1,\e} [h_x^j(\wdh x_t, \wdh u_t) x_t^{1,\e}]}_{L^2(\La)} dt +o(\e).
\earray\eeq

Therefore, \eqref{E2Q-1} and \eqref{E2Q-2} together with
the definition of $\EE_2^\QQ$ in \eqref{var-cost2} yield
that
\beq{EQ2}\barray
\EE_2^\QQ
\ad = \EE^\QQ \int_0^T\int_\La \wdh \rho_t^{-1}(\rho_t^{1,\e}+\rho_t^{2,\e})\ell(\wdh x_t(\la), \wdh u_t) d\la dt+ \wdh \rho_T^{-1}(\rho_T^{1,\e}+\rho_T^{1,\e}) m(\wdh x_T(\la))d\la \\
\aad + \EE^\QQ \int_0^T \int_\La \ell_x(\wdh x_t(\la), \wdh u_t) X_t^{1,\e}(\la) d\la dt +\int_\La  m_x(\wdh x_T(\la)) X_T^{1,\e}(\la)d\la  \\
\ad = \EE^\QQ \int_0^T \int_{\La} z_{2,t}^j h_x^j(\wdh x_t(\la), \wdh u_t) (x_t^{1,\e}(\la)+x_t^{2,\e}(\la)) \\
\aad \qquad\qquad\qquad  + z_{2,t}^j [h^j(\wdh x_t(\la), \wdh u_t) -h^j(\wdh x_t(\la), \wdh u_t)] d\la dt \\
\aad +\EE^\QQ \int_0^T \int_{\La} \half z_{2,t}^j h_{xx}^j(\wdh x_t(\la), \wdh u_t(\la)) x_t^{1,\e}(\la) x_t^{1,\e}(\la) d\la dt \\
\aad + \EE^\QQ \int_0^T \qv{p_t, [g_x^j(\wdh x_t,\wdh u_t) x_t^{1,\e}][h_x^j(\wdh x_t, \wdh u_t) x_t^{1,\e}]}_{L^2(\La)}dt \\
\aad + \EE^\QQ \int_0^T \qv{q_{2,t}^j, [h_x^j(\wdh x_t, \wdh u_t) x_t^{1,\e}] x_t^{1,\e}}_{L^2(\La)} dt + o(\e).
\earray\eeq

Combining $\EE_1^\QQ$ in \eqref{var-cost2} and the above equality,
we have
\bea \EE^\QQ_{3}:= \ad
\EE^\QQ \int_0^T \int_\La [\ell_x(\wdh x_t(\la), \wdh u_t)+z_{2,t} h_x^j(\wdh x_t(\la), \wdh u_t)](x_t^{1,\e}(\la)+x_t^{2,\e}(\la))d\la dt \\
\aad + \EE\int_\La m_x(\wdh x_T(\la)) (x_T^{1,\e}(\la) + x_T^{2,\e}(\la)) d\la,
\eea
which follows from the sum of the first term of integrand integrated in the first and third lines of \eqref{var-cost2} and the first term of the integrand integrated in the third line of \eqref{EQ2}. We now deal with $\EE_{3}^\QQ$. Define $\wdt x_t^\e = x_t^{1,\e}+ x_t^{2,\e}$. Following \eqref{x1} and \eqref{x2}, we have
\beq{wdt-x}\barray
d\wdt x_t^\e \ad = [\Dl \wdt x_t^\e + (b-g^jh^j)_x(\wdh x_t, \wdh u_t)\wdt x_t^\e \\
\aad\qquad\qquad + \dl(b-g^j h^j)(t; u^\e_t)+\half (b-g^j h^j)_{xx}(\wdh x_t, \wdh u_t) x_t^{1,\e} x_t^{1,\e}]dt \\
\aad +[\sg_x(\wdh x_t, \wdh u_t)\wdt x_t^\e + \dl \sg(t; u^\e_t) +\dl \sg_x(t; u^\e_t) x_t^{1,\e} + \half \sg_{xx}(\wdh x_t, \wdh u_t) x_t^{1,\e} x_t^{1,\e}] dW_t \\
\aad + [g_x^j(\wdh x_t, \wdh u_t) \wdt x_t^\e + \dl g^j(t; u^\e_t) +\dl g_x^j(t; u^\e_t) x_t^{1,\e} + \half g_{xx}^j(\wdh x_t, \wdh u_t) x_t^{1,\e}  x_t^{1,\e} ] dY_t^j.
\earray\eeq

Applying the It\^{o} formula to $\qv{p_t, \wdt x_t^\e}_{L^2(\La)}$ for the variational solution of SPDEs in \cite{LR15} with $V:=H_0^1(\La) \times H_0^1(\La), H:=L^2(\La)\times L^2(\La)$, $F:H\to\rr, (x,y) \mapsto \qv{x,y}_{L^2(\La)}$, and taking integration and expectations, we have
\beq{EQ1+2}\barray
\EE^\QQ_{3}
\ad= \EE^\QQ \int_0^T \qv{p_t, \dl b(t;u^\e_t)-g^j(\wdh x_t, \wdh u_t) \dl h^j(t;u_t^\e) \\
\aad \qquad\qquad\qquad
+  \half (b-g^j h^j)_{xx}(\wdh x_t, \wdh u_t) x_t^{1,\e}x_t^{1,\e}}_{L^2(\La)}  dt \\
\aad \quad + \EE\int_0^T \qv{p_t, \half g_{xx}^j(\wdh x_t, \wdh u_t)h^j(\wdh x_t, \wdh u_t) x_t^{1,\e} x_t^{1,\e}} dt \\
\aad \quad +\EE^\QQ \int_0^T \qv{q_{1,t}, \dl \sg(t;u_t^\e) + \half \sg_{xx}(\wdh x_t, \wdh u_t) x_t^{1,\e} x_t^{1,\e}}_{L^2(\La)} dt \\
\aad\quad + \EE^\QQ \int_0^T \qv{q_{2,t}^j,\dl g^j(t;u^\e_t) + \half g_{xx}^j (\wdh x_t,\wdh u_t) x_t^{1,\e} x_t^{1,\e}}_{L^2(\La)}dt +o(\e),
\earray\eeq
where we used the fact that
\bea\ad\!\!\!\!\!\!\!
\EE^\QQ\Big[\int_0^T \qv{p_t, \dl g^j(t; u_t^\e) \dl h^j(t; u_t^\e)}_{L^2(\La)} dt + \int_0^T \qv{q_{1,t}, \dl \sg_x(\wdh x_t, \wdh u_t)x_t^{1,\e}}_{L^2(\La)}dt \Big] =o(\e), \\
\ad\!\!\!\!\!\!\! \EE^\QQ \Big[ \int_0^T\!\! \qv{p_t, \dl g_x^j(t;u^\e_t) h^j(\wdh x_t, \wdh u_t) x_t^{1,\e}}_{L^2(\La)} dt+ \int_0^T\!\! \qv{q_{2,t}^j, \dl g_x^j(\wdh x_t,\wdh u_t) x_t^{1,\e}}_{L^2(\La)} dt\Big] =o(\e).
\eea

Now combining \eqref{E2Q-1}, \eqref{EQ2}, and \eqref{EQ1+2} yields
\beq{EQ12}\barray \ad\!\!\!
\EE_1^\QQ + \EE_2^\QQ\\
\aad\!\!\!= \EE^\QQ \int_0^T \int_{\La}  \dl \ell(t;u_t^\e) + \half \ell_{xx}(\wdh x_t(\la), \wdh u_t) x_t^{1,\e}(\la) x_t^{1,\e}(\la) d\la dt \\
\aad\!\!\!+ \EE^\QQ \int_\La \half m_{xx}(\wdh x_T(\la)) x_T^{1,\e} x_T^{1,\e}(\la) d\la \\
\aad\!\!\!+ \EE^\QQ \int_0^T  \int_\La \Big[ z_{2,t}^j \dl h^j(t; u^\e_t) + \half z_{2,t}^j h_{xx}^j(\wdh x_t(\la), \wdh u_t) x_t^{1,\e} (\la) x_t^{1,\e}(\la)\Big]  d\la dt  \\
\aad\!\!\!+ \EE^\QQ \int_0^T \qv{p_t, [g_x^j(\wdh x_t, \wdh u_t) x_t^{1,\e}] [h_x^j(\wdh x_t, \wdh u_t) x_t^{1,\e}]}_{L^2(\La)}dt\\
\aad\!\!\!+\EE^\QQ \int_0^T \qv{q_{2,t}^j, x_t^{1,\e} [h_x^j(\wdh x_t, \wdh u_t) x_t^{1,\e}]}_{L^2(\La)}dt \\
\aad\!\!\!+\EE^\QQ \int_0^T \qv{p_t, \dl b(t;u_t^\e)- g^j(\wdh x_t, \wdh u_t) \dl h^j(t; u_t^\e) + \half (b-g^j h^j)_{xx}(\wdh x_t, \wdh u_t) x_t^{1,\e} x_t^{1,\e}}_{L^2(\La)} dt \\
\aad\!\!\!+ \EE^\QQ \int_0^T \qv{p_t, \half g_{xx}^j(\wdh x_t, \wdh u_t) h^j(\wdh x_t, \wdh u_t) x_t^{1,\e} x_t^{1,\e}} dt \\
\aad\!\!\! + \EE^\QQ \int_0^T \qv{q_{1,t}, \dl \sg(t;u^\e_t) + \half \sg_{xx}(\wdh x_t, \wdh u_t) x_t^{1,\e} x_t^{1,\e}}dt \\
\aad\!\!\!+ \EE^\QQ \int_0^T \qv{q_{2,t}^j, \dl g^j(t;u_t^\e) + \half g_{xx}^j(\wdh x_t, \wdh u_t) x_t^{1,\e} x_t^{1,\e}}_{L^2(\La)}dt +o(\e).
\earray\eeq

Recalling the definition of the Hamiltonian in \eqref{def-H}, we further have
\beq{EQ12-cont}\barray
\ad \EE_1^\QQ + \EE_2^\QQ \\
\ad = \EE^\QQ \int_0^T H(\wdh x_t, u_t^\e, p_t, q_{1,t}, q_{2,t}, z_{2,t}-\qv{p_t, g(\wdh x_t, \wdh u_t)}_{L^2(\La)} \\
\aad \qquad \qquad - H(\wdh x_t, \wdh u_t, p_t, q_{1,t}, q_{2,t}, z_{2,t}-\qv{p_t, g(\wdh x_t, \wdh u_t)}_{L^2(\La)}) dt\\
\aad  + \EE^\QQ \int_0^T \int_{\La} \Big[\half b_{xx}(\wdh x_t, \wdh u_t)+ \half \ell_{xx}(\wdh x_t, \wdh u_t)\Big] x_t^{1,\e}(\la) x_t^{1,\e}(\la) d\la dt \\
\aad + \EE^\QQ \int_0^T\int_{\La} \half \big(z_{2,t}^j -\qv{p_t, g^j(\wdh x_t, \wdh u_t)}_{L^2(\La)}\big) h_{xx}^j(\wdh x_t, \wdh u_t) x_t^{1,\e}(\la) x_t^{1,\e}(\la) d\la dt  \\
\aad + \EE^\QQ \int_0^T \int_{\La} \Big[\half q_{1,t}(\la) \sg_{xx}(\wdh x_t(\la), \wdh u_t) +\half q_{2,t}^j(\la) g_{xx}^j(\wdh x_t(\la), \wdh u_t) \Big]x_t^{1,\e}(\la) x_t^{1,\e}(\la)  d\la dt \\
\aad + \EE^\QQ \int_{\La} \half  m_{xx}(\wdh x_T(\la)) x_T^{1,\e}(\la) x_T^{1,\e}(\la)d \la  + \EE^\QQ \int_0^T \qv{q_{2,t}^j, x_t^{1,\e} h_x^j(\wdh x_t, \wdh u_t) x_t^{1,\e}} dt + o(\e) \\
\ad =  \EE^\QQ \int_0^T H(\wdh x_t, u_t^\e, p_t, q_{1,t}, q_{2,t}, z_{2,t}-\qv{p_t, g(\wdh x_t, \wdh u_t)}_{L^2(\La)}) \\
\aad \qquad\qquad - H(\wdh x_t, \wdh u_t, p_t, q_{1,t}, q_{2,t}, z_{2,t} -\qv{p_t, g(\wdh x_t, \wdh u_t)}_{L^2(\La)}) dt \\
\aad\quad + \half \EE^\QQ \int_0^T \int_\La H_{xx}(\wdh x_t(\la), \wdh u_t, p_t(\la), q_{1,t}(\la), q_{2,t}(\la), z_{2,t}-\qv{p_t, g(\wdh x_t, \wdh u_t)}_{L^2(\La)})\\
\aad \qquad\qquad\qquad \qquad   \times x_t^{1,\e}(\la) x_t^{1,\e}(\la) d\la dt  \\
\aad \quad + \EE^\QQ \int_0^T\int_\La q_{2,t}^j(\la) h_x^j(\wdh x_t(\la), \wdh u_t) x_t^{1,\e}(\la) x_t^{1,\e}(\la) d\la dt \\
\aad \quad +\half \EE^\QQ \int_0^T m_{xx}(\wdh x_T(\la)) x_T^{1,\e}(\la) x_T^{1,\e}(\la) d\la + o(\e).
\earray\eeq
We proceed to 
characterize the adjoint equation for the quadratic term $x_t^{1,\e}(\la) x_t^{1,\e}(\la)$. We take the approach in \cite{SW21} to deal with the quadratic term $x_t^{1,\e} x_t^{1,\e}$ in \eqref{EQ12-cont} by making a bilinear form into a linear form on the tensor product $L^2(\La) \otimes L^2(\La)\cong L^2(\La^2)$.

\subsection{Mollified second order adjoint state}
Define $X_t^\e(\la,\mu):= x_t^{1,\e}(\la) x_t^{1,\e}(\mu)$. For $f_1, f_2 \in H_0^1(\La)$, applying It\^{o} formula for real-valued semi-martingales
\bea \qv{X_t^\e, f_1 \otimes f_2}_{L^2(\La^2)}:=\qv{x_t^{1,\e},f_1}_{L^2(\La)} \qv{x_t^{1,\e}, f_2}_{L^2(\La)}
\eea
and a density argument, we have in $L^2(\La^2)$,
\beq{dX-lamu}\barray\ad
dX_t^\e(\la, \mu) \\
\ad = x_t^{1,\e}(\la) d x_t^{1,\e}(\mu) + x_t^{1,\e}(\mu) dx_t^{1,\e}(\la) + d\qv{x_t^{1,\e}(\la), x_t^{1\e}(\mu)}_t \\
\ad = x_t^{1,\e}(\la) (\Dl_\mu  x_t^{1,\e}(\mu)+ (b-g^j h^j)_x (t,\mu; u_t^\e) x_t^{1,\e}(\mu)) dt \\
\aad + x_t^{1,\e}(\la) [g_x^j(\wdh x_t(\mu), \wdh u_t) h^j(\wdh x_t(\mu), \wdh u_t) x_t^{1,\e}(\mu) + \dl g^j(t,\mu; u_t^\e) h^j(\wdh x_t(\mu), \wdh u_t)]  dt \\
\aad+ x_t^{1,\e}(\la) [\sg_x(\wdh x_t(\mu),\wdh u_t) \wdh x_t^{1,\e}(\mu) + \dl \sg(t,\mu; u_t^\e)] dW_t \\
\aad + x_t^{1,\e}(\la) [g_x^j(\wdh x_t(\mu), \wdh u_t) x_t^{1,\e}(\mu) + \dl g^j(t,\mu; u_t^\e)] dB_t^j \\
\aad + x_t^{1,\e}(\mu
) (\Dl_\la x_t^{1,\e}(\la) + (b-g^j h^j)_x (t,\la ; u_t^\e) x_t^{1,\e}(\la)) dt \\
\aad + x_t^{1,\e}(\mu) [g_x^j(\wdh x_t(\la), \wdh u_t) h^j(\wdh x_t(\la), \wdh u_t) x_t^{1,\e}(\la)  + \dl g^j(t,\la; u_t^\e) h^j(\wdh x_t(\la), \wdh u_t) ] dt \\
\aad + x_t^{1,\e}(\mu) [\sg_x(\wdh x_t(\la), \wdh u_t) x_t^{1,\e}(\la) + \dl \sg(t,\la; u_t^\e)] dW_t \\
\aad + x_t^{1,\e}(\mu) [g_x^j(\wdh x_t(\la), \wdh u_t)  x_t^{1,\e}(\la) + \dl g^j(t,\la; u_t^\e)] dB_t^j \\
\aad +\qv{\sg_x(\wdh x_t(\la), \wdh u_t) x_t^{1,\e}(\la) + \dl \sg (t,\la; u_t^\e), \sg_x(\wdh x_t(\mu), \wdh u_t) x_t^{1,\e}(\mu) + \dl \sg(t,\mu; u_t^\e)}_{\CL_2(\Xi;\rr)} dt  \\
\aad + \qv{g_x^j(\wdh x_t(\la), \wdh u_t) x_t^{1,\e}(\la) + \dl g^j(t,\la; u_t^\e), g_x^j(\wdh x_t(\mu), \wdh u_t)x_t^{1,\e}(\mu) +\dl g^j(t,\mu; u_t^\e)}_{\rr} dt.
\earray\eeq

Rearranging terms and observing that
\bea
x_t^{1,\e}(\la)  \Dl_\mu x_t^{1,\e}(\mu) + x_t^{1,\e}(\mu) \Dl_\la x_t^{1,\e}(\la) =\Dl X_t^\e (\la, \mu),
\eea
we arrive at the following proposition.

\begin{prop}
The process $X_t^\e (\la,\mu)$ satisfies
\beq{X-lamu}
\left\{\barray
d X_t^\e (\la, \mu) \ad \!\!\!\!=\Big\{\Dl X_t^\e(\la, \mu) + \big[(b-g^j h^j)_x(\wdh x_t(\la), \wdh u_t) \\
\aad\quad + (b-g^j h^j)_x(\wdh x_t(\mu), \wdh u_t) + \qv{\sg_x(\wdh x_t(\la), \wdh u_t), \sg_x(\wdh x_t(\mu), \wdh u_t)}_{\CL_2(\Xi;\rr)}  \\
\aad\quad + \qv{g_x^j(\wdh x_t(\la), \wdh u_t), g_x^j(\wdh x_t(\mu), \wdh u_t)}_{\rr} + (g_x^j(\wdh x_t(\la), \wdh u_t) h^j(\wdh x_t(\la), \wdh u_t) \\
\aad\quad + g_x^j(\wdh x_t(\mu), \wdh u_t) h^j(\wdh x_t(\mu), \wdh u_t)) \big] X_t^\e(\la, \mu) + \Phi_t^\e(\la, \mu) \Big\} dt \\
\aad + \big[(\sg_x(\wdh x_t(\la), \wdh u_t) + \sg_x(\wdh x_t(\mu), \wdh u_t)) X_t^\e(\la,\mu) + \Psi_{1,t}^{\e}(\la,\mu) \big] dW_t \\
\aad + \big[ (g_x^j(\wdh x_t(\la), \wdh u_t) + g_x^j(\wdh x_t(\mu), \wdh u_t) ) X_t^\e(\la,\mu) + \Psi_{2,t}^{j, \e}(\la, \mu) \big] dB_t^j \\
X_0^\e(\la,\mu)\ad \!\!\!\! =0,
\earray\right.
\eeq
where $X^\e\in L^2([0,T]\times \Omega; H_0^1(\La^2)) \cap L^2(\Omega; C([0,T]; L^2(\La^2)))$ and $(\Phi^\e, \Psi_1^\e, \Psi_2^{j,\e}) \in L^2([0,T] \times \Omega; L^2(\La^2)) \times L^2([0,T]\times \Omega; \CL_2(\Xi; L^2(\La^2))) \times L^2([0,T]\times \Omega; L^2(\La^2)), j=1,\dots d$ are defined by
\bea
\Phi_t^\e(\la,\mu) \ad := \dl g^j(t,\la; u_t^\e) h^j(\wdh x_t(\la), \wdh u_t) x_t^{1,\e}(\la) + \dl g^j(t,\mu; u_t^\e ) h^j(\wdh x_t(\mu), \wdh u_t) x_t^{1,\e}(\mu)  \\
\aad + \qv{\sg_x (\wdh x_t(\la), \wdh u_t)x_t^{1,\e}(\la), \dl \sg(t,\mu; u_t^\e)}_{\CL_2(\Xi;\rr)} \\
\aad + \qv{\sg_x (\wdh x_t(\mu), \wdh u_t)  x_t^{1,\e}(\mu), \dl \sg(t,\la; u_t^\e)}_{\CL_2(\Xi;\rr)} \\
\aad +\qv{\dl \sg(t,\la; u_t^\e), \dl \sg(t,\mu; u_t^\e)}_{\CL_2(\Xi; \rr)} \\
\aad + \qv{g_x^j(\wdh x_t(\la), \wdh u_t) x_t^{1,\e}(\la), \dl g^j(t,\mu; u_t^\e)}_{\rr} \\
\aad + \qv{g_x^j(\wdh x_t(\mu), \wdh u_t) x_t^{1,\e}(\mu), \dl g^j(t,\mu; u_t^\e)}_{\rr} \\
\aad + \qv{\dl g^j(\wdh x_t(\la), \wdh u_t), \dl g^j(\wdh x_t(\mu), \wdh u_t) }_{\rr},
\eea
and $
\Psi_{1,t}^{\e}(\la, \mu) := x_t^{1,\e}(\la) \dl \sg(t,\mu; u_t^\e) + x_t^{1,\e}(\mu) \dl \sg(t,\la; u_t^\e)$, $\Psi_{2,t}^{j,\e} (\la, \mu) := x_t^{1,\e}(\la)\times$\\
$\dl g^j(t,\mu; u_t^\e)+ x_t^{1,\e}(\mu) \dl g^j(t,\la; u_t^\e)$.
\end{prop}

\begin{proof}
The derivation of \eqref{X-lamu} follows from \eqref{dX-lamu} and the corresponding regularity follows from the regularity of $x_t^{1,\e}$. The proof is complete.
\end{proof}

Let us define the operator $\pi: H_0^1(\La^2)  \to L^2(\La)$ as $\pi(\omega)(\la): = \omega(\la,\la)$, then \eqref{EQ12-cont} can be rewritten as
\beq{EQ12+}\barray \ad
\EE_1^\QQ + \EE_2^\QQ \\
\aad = \EE^\QQ \int_0^T H(\wdh x_t, u_t^\e, p_t, q_{1,t}, q_{2,t}, z_{2,t}-\qv{p_t, g(\wdh x_t,\wdh u_t)}_{L^2(\La)}) \\
\aad \qquad\qquad - H(\wdh x_t, \wdh u_t, p_t, q_{1,t}, q_{2,t}, z_{2,t} -\qv{p_t, g(\wdh x_t, \wdh u_t)}_{L^2(\La)}) dt \\
\aad\quad + \half \EE^\QQ \int_0^T \int_\La H_{xx}(\wdh x_t(\la), \wdh u_t, p_t(\la), q_{1,t}(\la), q_{2,t}(\la), z_{2,t} - \qv{p_t, g(\wdh x_t, \wdh u_t)}_{L^2(\La)}) \\
\aad \qquad \qquad\qquad \qquad  \times \pi(X_t^\e) (\la) d\la dt  \\
\aad\quad +
\EE^\QQ \int_0^T \int_\La q_{2,t}^j(\la) h_x^j(\wdh x_t(\la), \wdh u_t) \pi(X_t^\e)(\la) d\la dt \\
\aad \quad + \half \EE^\QQ \int_{\La} m_{xx}(\wdh x_T(\la)) X_T^{\e}(\la,\la) d\la.
\earray\eeq

By the trace theorem in \cite[Section 7.38]{AF03},
the operator $\pi: H_0^1(\La^2) \to L^2(\La)$ is continuous. Due to $X^\e \in L^2([0,T]\times \Omega; H_0^1(\La^2))$, terms in the fourth and six lines of \eqref{EQ12+} are
linear and bounded in $X^\e$. However, the spatial regularity of $X_T^\e \in L^2(\Omega; L^2(\La^2))$ does not guarantee that the last line of \eqref{EQ12+} is continuous in $X_T^\e$, since it involves evaluating $X_T^\e$ on the diagonal $(\la,\la)$.
To overcome this challenge, we use the method in \cite{SW21} to introduce a mollified terminal condition using the heat kernel. We define
\beq{m-eta}
m_{xx}^\eta(\la,\mu):=\half\big[m_{xx}(\wdh x_t(\la)) + m_{xx}(\wdh x_t(\mu))\big]  \frac{1}{\sqrt{4\pi \eta}}\exp \Big(-\frac{|\la-\mu|^2}{4 \eta} \Big) \in L^2(\La^2).
\eeq
Then we have
\beq{limit-m-eta}\barray\ad\!\!\!\!\!
\EE^\QQ \int_{\La} m_{xx}(\wdh x_T(\la)) x_T^{1,\e}(\la) x_T^{1,\e}(\la) d\la \\
\aad\!\!\!\!\!  =\lim_{\eta\to 0} \EE^\QQ \int_{\La^2} \half \big[ m_{xx}(\wdh x_t(\la)) + m_{xx}(\wdh x_t(\mu)) \big]x_T^{1,\e}(\la) x_T^{1,\e}(\mu) \\
\aad \qquad \qquad\qquad\quad  \times \frac{1}{\sqrt{4\pi \eta}} \exp\Big(-\frac{|\la-\mu|^2}{4\eta} \Big) d\la d\mu \\
\aad\!\!\!\!\! = \lim_{\eta\to 0} \EE^\QQ \int_{\La^2} \half \big[m_{xx}(\wdh x_t(\la)) + m_{xx}(\wdh x_t(\mu))\big] \frac{1}{\sqrt{4\pi \eta}} \exp\Big(-\frac{|\la-\mu|^2}{4\eta} \Big) X_T^\e(\la,\mu) d\la d\mu.
\earray\eeq

We now consider the following mollified  second order adjoint equation
\beq{P-eta}\left\{\barray
d P_t^\eta(\la,\mu) \ad = -\Big\{\Dl P_t^\eta(\la,\mu) + (b_x(\wdh x_t(\la), \wdh u_t) + b_x(\wdh x_t(\mu), \wdh u_t)) P_t^\eta(\la,\mu) \\
\aad \quad +\qv{\sg_x(\wdh x_t(\la), \wdh u_t), \sg_x(\wdh x_t(\mu), \wdh u_t)}_{\CL_2(\Xi;\rr)} P_t^\eta(\la,\mu) \\
\aad \quad + \qv{g_x^j(\wdh x_t(\la), \wdh u_t), g_x^j(\wdh x_t(\mu), \wdh u_t)}_{\rr} P_t^\eta(\la,\mu) \\
\aad \quad - (g^j(\wdh x_t(\la), \wdh u_t) h_x^j (\wdh x_t(\la), \wdh u_t)+g^j(\wdh x_t(\mu), \wdh u_t) h_x^j(\wdh x_t(\mu), \wdh u_t)) P_t^\eta(\la, \mu) \\
\aad \quad + \qv{\sg_x(\wdh x_t(\la), \wdh u_t) + \sg(\wdh x_t(\mu), \wdh u_t), Q_{1,t}^\eta (\la,\mu)}_{\CL_2(\Xi;\rr)} \\
\aad \quad + \qv{g_x^j(\wdh x_t(\la), \wdh u_t)+ g_x^j(\wdh x_t(\mu),\wdh u_t), Q_{2,t}^{j,\eta}(\la,\mu)}_{\rr} \\
\aad \quad + \pi^* ( H_{xx}(t,\wdh x_t, \wdh u_t, p_t, q_{1,t}, q_{2,t}, z_{2,t} -\qv{p_t, g(\wdh x_t, \wdh u_t)}_{L^2(\La)}) ) \\
\aad \quad + \pi^* (q_{2,t}^j(\la) h_x^j(\wdh x_t(\la), \wdh u_t)+ q_{2,t}^j(\mu) h_x^j(\wdh x_t(\mu), \wdh u_t))\Big\} dt \\
\aad \quad + Q_{1,t}^\eta(\la,\mu)dW_t + Q_{2,t}^{j,\eta}(\la,\mu) dB_t^j \\
P_T^\eta(\la, \mu) \ad = m_{xx}^\eta (\la,\mu),
\earray\right.
\eeq
where $m_{xx}^\eta$ is given by \eqref{m-eta}, and $\pi^*: L^2(\La)\to H^{-1}(\La^2)$ is the adjoint of the operator of $\pi$ given by
\beq{pi-star}
\qv{\pi^* f, w}_{H^{-1}(\La^2)\times H_0^1(\La^2)} := \int_{\La} f(\la) \pi(w)(\la) d\la = \int_{\La} f(\la) w(\la,\la)d\la,
\eeq
for $f \in L^2(\La), w\in H_0^1(\La^2)$. Then one can show that \eqref{P-eta} has a unique variational solution $(P^\eta, Q_1^\eta, Q_2^{\eta})$ on the Gelfand triple $H_0^1(\La^2) \hookrightarrow L^2(\La^2) \hookrightarrow H^{-1}(\La^2)$ such that
\bea
P^\eta \in L^2([0,T]\times \Omega; H_0^1(\La^2)) \cap L^2(\Omega; C([0,T]; L^2(\La^2)))
\eea
and
\bea
Q_1^\eta \in L^2([0,T]\times \Omega; \CL_2(\Xi; L^2(\La^2))), \quad  Q_2^{\eta,j} \in L^2([0,T]\times \Omega; L^2(\La^2)), j=1,\dots d.
\eea
We refer to the work of \cite{Ben83,SW21,LR15} for the existence and uniqueness of the variational solutions on the Gelfand triple.

Now, applying
the It\^{o} formula to $\qv{P_t^\eta(\la,\mu), X_t^\e(\la,\mu)}_{\rr}$,
we have
\beq{P-X}\barray\ad\!\!\!
d\qv{P_t^\eta(\la,\mu), X_t^\e(\la, \mu)}_{\rr} \\
\aad\!\!\! = \qv{P_t^\eta(\la,\mu), dX_t^\e(\la,\mu)}_{\rr} + \qv{X_t^\e(\la,\mu), dP_t^\eta(\la,\mu)}_{\rr} + d\qv{P^\eta(\la,\mu), X^\e(\la,\mu)}_t \\
\aad\!\!\! = \bqv{P_t^\eta(\la,\mu),  \Dl X_t^\e(\la,\mu) + (b-g^j h^j)_x(\wdh x_t(\la), \wdh u_t)X_t^\e(\la,\mu)}_{\rr} \\
\aad + \qv{P_t^\eta(\la,\mu), (b-g^j h^j)_x(\wdh x_t(\mu), \wdh u_t) X_t^\e(\la,\mu)}_{\rr} dt \\
\aad + \bqv{P_t^\eta(\la,\mu), \qv{\sg_x(\wdh x_t(\la), \wdh u_t), \sg_x(\wdh x_t(\mu), \wdh u_t)}_{\CL_2(\Xi;\rr)}  X_t^\e(\la,\mu)}_{\rr} dt \\
\aad + \bqv{P_t^\eta(\la,\mu), \qv{g_x^j(\wdh x_t(\la), \wdh u_t), g_x^j(\wdh x_t(\mu), \wdh u_t)}_{\rr} X_t^\e(\la, \mu)}_{\rr}dt  \\
\aad + \bqv{P_t^\eta(\la,\mu),[ g_x^j(\wdh x_t(\la, \wdh u_t)) h^j(\wdh x_t(\la), \wdh u_t) + g_x^j(\wdh x_t(\mu), \wdh u_t) h^j(\wdh x_t(\mu), \wdh u_t)] X_t^\e(\la,\mu)}_{\rr} dt \\
\aad +\qv{P_t^\eta(\la,\mu), \Phi_t^\e(\la,\mu) }_{\rr} dt \\
\aad + \bqv{P_t^\eta(\la,\mu), \big[(\sg_x(\wdh x_t(\la), \wdh u_t) + \sg_x(\wdh x_t(\mu), \wdh u_t)) X_t^\e(\la,\mu) + \Psi_{1,t}^{\e}(\la,\mu) \big] dW_t}_{\rr}  \\
\aad + \bqv{P_t^\eta(\la,\mu),\big[(g_x^j(\wdh x_t(\la), \wdh u_t) + g_x^j(\wdh x_t(\mu), \wdh u_t)) X_t^\e(\la,\mu) + \Psi_{2,t}^{j, \e}(\la,\mu)\big] dB_t^j}_{\rr} \\
\aad + \bqv{X_t^\e(\la, \mu), -\big \{\Dl P_t^\eta(\la,\mu)  + (b_x(\wdh x_t(\la), \wdh u_t)+ b_x(\wdh x_t(\mu), \wdh u_t)) P_t^\eta(\la,\mu) \\
\aad\qquad \qquad\qquad + \qv{\sg_x(\wdh x_t(\la), \wdh u_t), \sg_x(\wdh x_t(\mu), \wdh u_t)}_{\CL_2(\Xi;\rr)} P_t^\eta(\la,\mu) \\
\aad \qquad \qquad\qquad + \qv{g_x^j(\wdh x_t(\la), \wdh u_t), g_x^j(\wdh x_t(\mu), \wdh u_t)}_{\rr} P_t^\eta(\la,\mu) \\
\aad \qquad \qquad\qquad- \big[g^j(\wdh x_t(\la), \wdh u_t) h_x^j(\wdh x_t(\la), \wdh u_t) + g^j(\wdh x_t(\mu), \wdh u_t) h_x^j(\wdh x_t(\mu), \wdh u_t)\big]  P_t^\eta(\la,\mu) \\
\aad \qquad \qquad\qquad + \qv{\sg_x(\wdh x_t(\la), \wdh u_t)+ \sg_x(\wdh x_t(\mu), \wdh u_t), Q_1^{2,\eta}(\la,\mu)}_{\CL_2(\Xi;\rr)} \\
\aad \qquad \qquad\qquad + \qv{g_x^j(\wdh x_t(\la), \wdh u_t) + g_x^j(\wdh x_t(\mu), \wdh u_t), Q_{2,t}^{j,\eta}(\la, \mu)}_{\rr} \\
\aad \qquad \qquad\qquad + \pi^* (H_{xx}(\wdh x_t(\la), \wdh u_t, p_t, q_{1,t}, q_{2,t}, z_{2,t}- \qv{p_t, g(\wdh x_t, \wdh u_t)}_{L^2(\La)})) \\
\aad \qquad \qquad\qquad + \pi^* (q_{2,t}^j(\la) h_x^j(\wdh x_t(\la), \wdh u_t) + q_{2,t}^j(\mu) h_x^j(\wdh x_t(\mu), \wdh u_t)) \big\} }_{\rr}  dt \\
\aad \quad + \qv{X_t^\e(\la,\mu), Q_{1,t}^\eta(\la,\mu) dW_t + Q_{2,t}^{j,\eta}(\la,\mu) dB_t^j}_{\rr} \\
\aad \quad + \qv{(\sg_x(\wdh x_t(\la), \wdh u_t) + \sg_x(\wdh x_t(\mu), \wdh u_t)) X_t^\e(\la,\mu) + \Psi_{1,t}^{\e}(\la, \mu), Q_{1,t}^\eta(\la,\mu)}_{\rr} dt \\
\aad \quad + \qv{(g_x^j(\wdh x_t(\la), \wdh u_t) + g_x^j(\wdh x_t(\mu), \wdh u_t)) X_t^\e(\la,\mu) + \Psi_{2,t}^{j,\e}(\la,\mu), Q_{2,t}^{j,\eta}(\la,\mu)}_{\rr} dt.
\earray\eeq

Consequently, we have
\bea\ad
d\qv{P_t^\eta(\la,\mu), X_t^\e(\la,\mu)}_{\rr} \\
\aad = \qv{P_t^\eta(\la, \mu), \Phi_t^\e(\la,\mu) }_{\rr} dt  \\
\aad \quad -\bqv{X_t^\e(\la,\mu), \pi^* \big(H_{xx}(t,\wdh x_t, \wdh u_t, p_t, q_{1,t}, q_{2,t}, z_{2,t} -\qv{p_t, g(\wdh x_t, \wdh u_t)}_{L^2(\La)})\big) }_{\rr} dt \\
\aad \quad -\bqv{ X_t^\e(\la,\mu),  \pi^* (q_{2,t}^j(\la) h_x^j(\wdh x_t(\la), \wdh u_t) + q_{2,t}^j (\mu) h_x^j(\wdh x_t(\mu), \wdh u_t))}_{\rr} dt \\
\aad \quad + \qv{Q_{1,t}^\eta(\la,\mu), \Psi_{1,t}^{\e}(\la, \mu)}_{\CL_2(\Xi;\rr)} dt + \qv{Q_{2,t}^{j,\eta}(\la, \mu), \Psi_{2,t}^{j,\e}(\la,\mu) }_{\rr} dt  \\
\aad\quad + \bqv{P_t^\eta(\la,\mu),\big[(\sg_x(\wdh x_t(\la), \wdh u_t) + \sg_x(\wdh x_t(\mu), \wdh u_t)) X_t^\e(\la,\mu) + \Psi_{1,t}^{\e}(\la,\mu) \big] dW_t}_{\rr}  \\
\aad \quad + \bqv{P_t^\eta(\la,\mu),\big[(g_x^j(\wdh x_t(\la), \wdh u_t) + g_x^j(\wdh x_t(\mu), \wdh u_t)) X_t^\e(\la,\mu) + \Psi_{2,t}^{j, \e}(\la,\mu)\big] dB_t^j}_{\rr} \\
\aad \quad + \qv{X_t^\e(\la,\mu), Q_{1,t}^\eta(\la,\mu) dW_t + Q_{2,t}^{j,\eta}(\la,\mu)dB_t^j}_{\rr}.
\eea

Therefore, taking integration and expectation and considering initial and terminal conditions of $X^\e(\la,\mu)$ and $P^\eta(\la,\mu)$, respectively, we obtain
\beq{int-PX}\barray\ad\!\!\!\!\!\!
\EE^\QQ \big[ \qv{m_{xx}^\eta, X_T^\e}_{L^2(\La^2)} \big] \\
\aad\!\!\!\!\!\!= \EE^\QQ \Big[ \int_0^T \qv{P_t^\eta, \Phi_t^\e}_{L^2(\La^2)} + \qv{Q_{1,t}^\eta, \Psi_{1,t}^{\e}}_{\CL_2(\Xi; L^2(\La^2))} + \qv{Q_{2,t}^{j,\eta}, \Psi_{2,t}^{j,\e}}_{L^2(\La^2)}  dt \Big] \\
\aad\!\!\!\!\!\!- \EE^\QQ \!\! \int_0^T\!\!\!\! \bqv{X_t^\e, \pi^*( H_{xx}(\wdh x_t, \wdh u_t, p_t, q_{1,t}, q_{2,t}, z_{2,t}-\qv{p_t, g(\wdh x_t, \wdh u_t)}_{L^2(\La)}))}_{H_0^1(\La^2), H^{-1}(\La^2)}  dt \\
\aad\!\!\!\!\!\!- \EE^\QQ \int_0^T \bqv{X_t^\e, \pi^* (q_{2,t}^j h_x^j(\wdh x_t, \wdh u_t) + q_{2,t}^j h_x^j(\wdh x_t, \wdh u_t) )}_{H_0^1(\La^2), H^{-1}(\La^2)} dt.
\earray\eeq

Consequently, we characterize the mollified second order adjoint equation by \eqref{P-eta}. Then combining \eqref{EQ12-cont} and \eqref{int-PX}, and noticing  \eqref{pi-star}, we get
\beq{CT-eta1} \barray\ad
\EE_1^\QQ + \EE_2^\QQ \\
\aad \!\!\!\! = \EE^\QQ \int_0^T H(\wdh x_t, u_t^\e, p_t, q_{1,t}, q_{2,t}, z_{2,t}-\qv{p_t, g(\wdh x_t, \wdh u_t)}_{L^2(\La)}) \\
\aad\qquad\qquad - H(\wdh x_t, \wdh u_t, p_t, q_{1,t}, q_{2,t}, z_{2,t} -\qv{p_t, g^j(\wdh x_t, \wdh u_t)}_{L^2(\La)}) dt\\
\aad +\half \EE^\QQ \int_0^T \qv{P_t^\eta, \Phi_t^\e}_{L^2(\La^2)} + \qv{Q_{1,t}^\eta, \Psi_{1,t}^{\e}}_{\CL_2(\Xi; L^2(\La^2))}  + \qv{Q_{2,t}^{j,\eta}, \Psi_{2,t}^{j,\e}}_{L^2(\La^2)} dt \\
\aad + \half \EE^\QQ \Big[ \int_{\La} m_{xx}(\wdh x_T(\la)) x_T^{1,\e}(\la) x_T^{1,\e}(\la) d\la  -  \int_\La m_{xx}^\eta(\wdh x_T(\la)) x_T^{1,\e}(\la) x_T^{1,\e}(\la) d\la \Big]  \\
\aad\!\!\!\!:= \EE_4^\QQ +\half \EE^\QQ \Big[\int_\La m_{xx}(\wdh x_T(\la)) x_T^{1,\e}(\la) x_T^{1,\e}(\la) d\la  -\int_\La m_{xx}^\eta(\wdh x_T(\la)) x_T^{1,\e}(\la) x_T^{1,\e}(\la) d\la \Big].
\earray\eeq

\subsection{The limit of the mollified second order adjoint equation}
We now in the position to study the limit of the mollified second order adjoint equation as $\eta \to 0$, i.e., $P=\lim_{\eta \to 0} P^\eta$.  Recall \eqref{m-eta}, \eqref{limit-m-eta}, and \eqref{pi-star}, $m_{xx}^\eta$ is chosen such that
\bea\ad
\lim_{\eta \to 0} m_{xx}^\eta = \pi^* ( m_{xx}(\wdh x_T)),  \quad \text{ in } H^{-1}(\La^2).
\eea

Letting $\eta \to 0$, we have
\begin{thm}\label{thm:PQ}
The equation
\beq{eta}\left\{ \barray
d P_t(\la,\mu) \ad = - \Big\{ \Dl P_t(\la,\mu) + (b_x(\wdh x_t(\la), \wdh u_t)+ b_x(\wdh x_t(\mu), \wdh u_t)) P_t(\la, \mu)) \\
\aad  +  \qv{\sg_x(\wdh x_t(\la), \wdh u_t), \sg_x(\wdh x_t(\mu), \wdh u_t)}_{\CL_2(\Xi; \rr)} P_t(\la,\mu)  \\
\aad + \qv{g_x^j(\wdh x_t(\la), \wdh u_t),  g_x^j(\wdh x_t(\mu), \wdh u_t)}_{\rr} P_t(\la,\mu) \\
\aad - (g^j(\wdh x_t(\la), \wdh u_t) h_x^j(\wdh x_t(\la), \wdh u_t) + g^j(\wdh x_t(\mu), \wdh u_t) h_x^j(\wdh x_t(\mu), \wdh u_t)) P_t(\la, \mu) \\
\aad + \qv{\sg_x(\wdh x_t(\la), \wdh u_t) + \sg_x(\wdh x_t(\mu), \wdh u_t), Q_{1,t}(\la,\mu)}_{\CL_2(\Xi; \rr)} \\
\aad + \qv{g_x^j(\wdh x_t(\la), \wdh u_t) + g_x^j(\wdh x_t(\mu), \wdh \mu_t), Q_{2,t}^j(\la, \mu)}_{\rr} \\
\aad + \pi^* (H_{xx}(\wdh x_t(\la), \wdh u_t, p_t(\la), q_{1,t}(\la), q_{2,t}(\la), z_{2,t}-\qv{p_t, g(\wdh x_t, \wdh u_t)}_{L^2(\La)}) ) \\
\aad + \pi^* (q_{2,t}^j(\la) h_{x}^j(\wdh x_t(\la), \wdh u_t) + q_{2,t}^j (\mu)  h_x^j(\wdh x_t(\mu), \wdh u_t))
\Big\} dt \\
\aad+ Q_{1,t}(\la,\mu) d W_t + Q_{2,t}^j(\la, \mu) dB_t^j \\
P_T(\la,\mu) \ad = \pi^*(m_{xx}(\wdh x_T(\la)) ),
\earray\right.
\eeq
has a unique adapted solution $(P,Q_1, Q_2)$ with $Q_2:=(Q_2^1,\dots,Q_2^d)$ such that
\bea
P\in L^2([0,T]\times \Omega; L^2(\La^2)) \cap L^2([0,T]\times \Omega; C([0,T]; H^{-1}(\La^2)))
\eea
and
\bea
Q_1 \in L^2([0,T]\times \Omega; \CL_2(\Xi; H^{-1}(\La^2))), \quad Q_2^j \in L^2([0,T]\times \Omega, H^{-1}(\La^2)), j=1,\dots d.
\eea
\end{thm}
\begin{proof}
The proof is similar to \cite[Theorem 6.1]{SW21}, thus details are omitted.
\end{proof}

\subsection{Proof of \thmref{thm:Peng-smp}}
In order to prove the  stochastic maximum principle in \eqref{thm:Peng-smp}, we take limit $\e \to 0$ first and then let $\eta\to 0$. The justification of taking the limit $\e \to 0$ first is that if we take $\eta\to 0$ first, we cannot obtain needed asymptotics for the fourth line of \eqref{CT-eta1}  due to the lack of regularities for $Q_1$ and $Q_2$; see details in \cite[Section 7]{SW21}. To proceed, we need compactness of $x_T^{1,\e}, \e>0$ in $L^2(\La)$.

\begin{lem}\label{lem:compact}
For $\ga \in (0,1/2)$ and $\e \in (0,T-\tau)$, we have
\bea
\EE^\QQ \Big[\|x_T^{1,\e}\|_{H_0^\ga(\La)}^2 \Big] \leq K \e.
\eea
\end{lem}
\begin{proof}
Let $\wdt x_t^{1,\e}:= x_t^{1,\e}/\sqe, t\in [0,T]$. Then \eqref{x1} and \assref{ass} implies that there exists a unique mild solution $\wdt x_t^{1,\e}$ satisfying
\bea\!\!\!
\wdt x_t^{1,\e} \ad\!\!\!\! = \int_0^T\! e^{(t-s)\Dl} \Big[ (b_x-g^j h_x^j)(\wdh x_t, \wdh u_t) \wdt x_s^{1,\e} + \frac{1}{\sqe} h^j(\wdh x_s, \wdh u_s) (g^j(\wdh x_s, u_s^\e)- g^j(\wdh x_s, \wdh u_s)) \Big] ds \\
\aad \!\!\!\!+ \int_0^t e^{(t-s)\Dl} \big[ \sg_x(\wdh x_s, \wdh u_s)  \wdt x_s^{1,\e} + \frac{1}{\sqe} (\sg(\wdh x_s, u_s^\e)- \sg(\wdh x_s, \wdh u_s))\big]dW_s \\
\aad \!\!\!\! + \int_0^t e^{(t-s)\Dl} \big[ g_x^j(\wdh x_s, \wdh u_s) \wdt x_s^{1,\e} + \frac{1}{\sqe} (g^j(\wdh x_s, u_s^\e) - g^j(\wdh x_s, \wdh u_s)) \big] dB_s^j,
\eea
with $\wdt x_{0}^{1,\e}=0$. It is equivalently to say that it has a unique variational solution; see the equivalence between mild solutions and variational solutions in \cite{LR15}. We will use the mild solution to obtain estimates. Since $e^{t\Dl}$ is an analytic semigroup, we have
\beq{St-est}
\|e^{t \Dl} f\|_{H_0^\ga(\La)} \leq K t^{-\ga} \|f\|_{L^2(\La)}, \quad \forall f\in L^2(\La).
\eeq
Then the boundedness of $b_x, h_x$, and the linear growth of $g$ implies
\bea\ad
\EE^\QQ\Big[ \Big\| \int_0^T e^{(T-s)\Dl} (b_x- g^j h_x^j)(\wdh x_t, \wdh u_t) \wdt x_s^{1,\e} ds \Big \|_{H_0^\ga(\La)} ^2 \Big]  \\
\aad \leq K \sup_{t\in [0,T]} \EE^\QQ \big[ \|\wdt x_t^{1,\e} \|_{H_0^\ga(\La)}^2 \big] \int_0^T \frac{1}{(T-s)^{2\ga}}ds \\
\aad \qquad\quad  \times \Big(1+\sup_{t\in [0,T]}\EE^\QQ[\|\wdh x_t\|_{L^2(\La)}^2 ] + \sup_{t\in [0,T]} \EE^\QQ \|\wdh u_t\|_\CU^2 \Big) <\infty.
\eea
In the above, we used the fact that $\sup_{t\in [0,T]}\EE^\QQ[\|\wdt x_t^{1,\e}\|_{H_0^\ga(\La)}^2]<\infty$, which can be obtained by a similar argument as the proof of this lemma together with the application of the Grownall inequality. Thus, we omit the proof of this fact.

The definition of $u_t^\e$ in \eqref{ue}, the boundedness of $h$, and the linear growth of $g$ yield
\bea \ad
\EE^\QQ \Big[\Big\| \int_0^T e^{(T-s)\Dl} \frac{1}{\sqe} h^j(\wdh x_s, \wdh u_s) (g^j(\wdh x_s, u_s^\e)- g^j(\wdh x_s, \wdh u_s)) ds \Big\|_{H_0^\ga(\La)}^2 \Big] \\
\aad =  \frac{1}{\e} \EE^\QQ \int_\tau^{\tau+\e} \frac{1}{(T-s)^{2\ga}} \| h^j(\wdh x_s, \wdh u_s) (g^j(\wdh x_s, v)- g^j(\wdh x_s, \wdh u_s)) \|_{L^2(\La)}^2 ds  \\
\aad \leq K \Big(1+ \sup_{t\in [0,T]}\EE^\QQ [\|\wdh x_t\|_{L^2(\La)}^2] + \|v\|_{\CU}^2 + \sup_{t\in [0,T]}\EE^\QQ [\|\wdh u_t\|_{\CU}^2 ]\Big) <\infty.
\eea
For the stochastic integral, we have
\bea\ad
\EE^\QQ \Big[\Big\|\int_0^T e^{(T-s)\Dl} \sg_x(\wdh x_s, \wdh u_s) \wdt x_s^{1,\e} dW_s \Big\|_{H_0^\ga(\La)}^2 \Big] \\
\aad =\EE^\QQ \int_0^T \|e^{(T-s)\Dl} \sg_x(\wdh x_s, \wdh u_s) \wdt x_s^{1,\e} \|_{\CL_2(\Xi; H_0^\ga(\La))}^2 ds  \\
\aad \leq \EE^\QQ \int_0^T \|e^{(T-s)\Dl}\|_{\CL(L^2(\La); H_0^\ga(\La))}^2 \|\sg_x(\wdh x_s, \wdh u_s)\|_{\CL(L^2(\La); \CL_2(\Xi; L^2(\La)))}^2  \|\wdt x_t^{1,\e}\|_{L^2(\La)}^2 ds  \\
\aad \leq K \sup_{t\in [0,T]} \EE^\QQ \big[ \|\wdt x_t^{1,\e}\|_{L^2(\La)}^2 \big]\int_0^T (T-s)^{-2\ga} ds <\infty.
\eea
Moreover,
\bea\ad
\EE^\QQ \Big[\Big\| \frac{1}{\sqe} \int_0^T e^{(T-s)\Dl} ( \sg(\wdh x_s, u_s^\e)- \sg(\wdh x_s, \wdh u_s))  dW_s \Big\|_{H_0^\ga(\La)}^2 \Big] \\
\aad \leq \EE^\QQ \Big[\frac{1}{\e} \int_\tau^{\tau+\e} \| e^{(T-s)\Dl} (\sg(\wdh x_s, v) - \sg(\wdh x_s, \wdh u_s)) \|_{\CL_2(\Xi; H_0^\ga(\La))}^2 ds \Big] \\
\aad \leq K \Big(1+\sup_{t\in [0,T]} \EE^\QQ \big[\|\wdh x_t\|_{H_0^\ga(\La)}^2 \big] + \|v\|_\CU^2 + \sup_{t\in [0,T]} \EE^\QQ \big[\| \wdh u_t\|_{\CU}^2 \big]\Big) <\infty.
\eea
For estimates of stochastic integrals w.r.t.  Brownian motions $B$, we can apply the same argument as that of $W$. The proof of the lemma is then completed.
\end{proof}

We now proceed to prove our main result.

\begin{proof}[Proof of \thmref{thm:Peng-smp}]
Recall $\EE_4^\QQ$ in \eqref{CT-eta1}, we divide by $\e$ to $\EE_4^\QQ$ and let $\e \to 0$. Then \eqref{ue} implies that as $\e\to 0$,
\bea\disp
\frac{1}{\e}\EE_4^\QQ
\ad \to \EE^\QQ \Big[ H(\wdh x_\tau, v, p_\tau, q_{1,\tau}, q_{2,\tau}, z_{2,\tau}-\qv{p_\tau, g(\wdh x_\tau, \wdh u_\tau)}_{L^2(\La)}) \\
\aad \qquad\quad - H(\wdh x_\tau, \wdh u_\tau, p_\tau, q_{1,\tau}, q_{2,\tau}, z_{2,\tau} -\qv{p_\tau, g(\wdh x_\tau, \wdh u_\tau)}_{L^2(\La)}) \Big| \CF_\tau^Y \Big] \\
\aad\quad + \half \EE^\QQ \Big[\qv{P_\tau^\eta(\la,\mu), \bqv{\sg(\wdh x_\tau(\la), v) - \sg(\wdh x_\tau(\la),\wdh u_\tau), \\
\aad \qquad \qquad\qquad \qquad\quad \sg(\wdh x_\tau(\mu),v) - \sg(\wdh x_\tau(\mu), \wdh u_\tau)}_{\CL_2(\Xi;\rr)} } \Big| \CF_\tau^Y \Big] \\
\aad\quad  +\half  \EE^\QQ \Big[ \qv{P_\tau^\eta(\la,\mu), \bqv{g^j(\wdh x_\tau(\la), v)- g^j(\wdh x_\tau(\la), \wdh u_\tau),\\
\aad \qquad \qquad\qquad \qquad\quad  g^j(\wdh x_\tau(\mu),v) - g^j(\wdh x_\tau(\mu), \wdh u_\tau)}_{\rr} } \Big| \CF_\tau^Y \Big].
\eea

It remains to prove that
\bea\disp
\lim_{\eta \to 0} \lim_{\e\to 0} \frac{1}{\e}\EE^\QQ\Big[  \int_\La m_{xx}(\wdh x_T(\la)) x_T^{1,\e}(\la) x_T^{1,\e}(\la) d\la- \int_\La m_{xx}^\eta(\la,\mu) X_T^\e(\la,\mu) d\la d \mu  \Big]=0.
\eea
By \lemref{lem:compact} and the compact embedding $H_0^\ga(\La) \subset \subset L^2(\La), \ga\in (0,1/2)$, we can extract a subsequence of $\wdt x_T^{1,\e}$ such that it converges to $\wdt x_T^1 \in L^2(\La)$ in $L^2(\La)$. Thus, we have
\bea\ad
\lim_{\eta\to 0}\lim_{\e\to 0} \frac{1}{\e} \EE^\QQ \Big[\int_\La m_{xx}(\wdh x_T(\la)) x_T^{1,\e}(\la) x_T^{1,\e}(\la)d\la - \int_{\La^2} m_{xx}^\eta(\la,\mu) X_T^\e(\la,\mu)d\la d\mu \Big] \\
\aad =\lim_{\eta\to 0}\EE^\QQ \Big[\int_\La m_{xx}(\wdh x_T(\la)) \wdt x_T^1(\la) \wdt x_T^1(\la) d\la - \int_{\La^2} m_{xx}^\eta(\la,\mu) \wdt x_T^1(\la) \wdt x_T^1(\mu) d\la d\mu\Big]=0.
\eea
Consequently, we complete the proof.
\end{proof}

\section{Numerical algorithm for solving partially observed optimal control}\label{sec:algo}
In this section, we develop a numerical algorithm to solve the partially observed optimal control for SPDEs where the system state and observation process have correlated noise. In practical applications, controlling volatility is difficult. Therefore, in this section, we consider that the diffusion coefficients in \eqref{spde} are independent of the control variable.  Moreover, we note that the equation \eqref{pq} of $p_t$ has a series term in the drift due to
\beq{series}
\qv{\sg_x(x,u), q_{1,t}}_{\CL_2(\Xi; L^2(\La))} = \sum_{i=1}^\infty [\sg(x,u)e_i]_x^* q_{1,t}e_i.
\eeq
From computational point of view, it is best to consider the state to be driven either by the additive cylindrical Wiener process $W_t$ or by finitely many Brownian motions $W_t^i,i=1,\dots, N$, $N\in \NN$ with multiplicative coefficients. The correlated noise $B_t$ stays the same. Consequently, \eqref{series} becomes a finite sum in the drift of $p_t$.

The state process we consider becomes
\beq{spde-num}\left\{\barray
dx_t^u = [\Dl x_t^u + b(x_t^u, u_t)] dt + \sg^i(x_t^u) dW_t^i + g^j(x_t^u) dB_t^j \\
x_0^u = x_0.
\earray\right.\eeq
Again, we used the Einstein summation convention.
The observation process is
\bea
d Y_t = h(x_t^u,u_t) dt + dB_t,
\eea
where $h$ is a mapping from $L^2(\La) \times U$ to $\rr^d$. We assume that the control is in a convex set. \thmref{thm:pon-smp} implies that the gradient of cost functional with respect to control variable becomes
\beq{nabla-J}
\nabla_u J(u):=\EE^\QQ \big[b_u^*(x_t, u_t) p_t + L_u(x_t, u_t)+ h_u^{j,*}(x_t, u_t) z_{2,t}^j | \CF_t^Y \big].
\eeq

\subsection{Conditional SGD algorithm}
We are interested in computing the optimal control $\wdh u_t$ at time $t$ given the available information $\CF_t^Y$. We use the gradient descent algorithm to achieve this purpose:
\beq{SGD}
u_t^{\iota+1,Y}= u_t^{\iota,Y}-\al \nabla_u J(u_t^{\iota,Y}),\iota = 0,1,2,\dots
\eeq
where $\al$ is learning rate or step size, $\iota$ is the iteration number. We observe that computing $\nabla_u J(u_t^{\iota,Y})$ in \eqref{SGD} requires the value $p_t$ at time $t$. To obtain $p_t$, we need to solve the backward SPDE \eqref{pq} and the backward SDE \eqref{z} from time $T$ to $t$. However, solving \eqref{pq} depends on the control values at time $r\in (t,T]$. In the SGD update \eqref{SGD} at time $t$, we only have access to the filtration $\CF_t^Y$, while $\CF_r^Y$ for $r\in (t,T]$ remains unavailable. Hence, the control $u_r$ cannot be used directly in solving \eqref{pq} for $r\in [t,T]$ given the available information $\CF_t^Y$.

To overcome this challenge, we replace $u_r^{\iota,Y}, r\in (t,T]$ with its conditional expectation given $\CF_t^Y$. This substitution is justified by the fact that the conditional expectation provides the best approximation of $u_r^{\iota,Y}$ based on the available information $\CF_t^Y$. Define
\bea\ad
u_r^{\iota,Y}|_t := \EE^\QQ[u_r^{\iota,Y}| \CF_t^Y].
\eea

We then arrive at the following algorithm
\beq{cond-SGD}
u_r^{\iota+1,Y}|_t = u_r^{\iota,Y}|_t -\al \nabla_u J(u_r^{\iota,Y}|_t), \iota=0,1,2,\dots,
\eeq
for $r\in [t,T]$ and $\nabla_u J(u_r^{\iota,Y})$ is calculated by \eqref{nabla-J}.

In numerical approximations, all values in the Hilbert space, namely $x_t, u_t, p_t, q_1^i$, $q_2^j$, $i=1,\dots, N, j=1,\dots, d$, will be projected to the finite-dimensional space $S_h$ using the finite element approximation. To avoid introducing more notation, in the following, we regard them as $S_h$-valued variables with certain abuse of notation. We will apply the implicit Euler method for time discretization to avoid restrictions between time and space step sizes. For details on numerical schemes for solving forward-backward SPDEs of $x$ and $p$ using finite element approximations and those for solving the BSDE of $z$, we refer to \cite[Section 3,4]{BCQ25}.

Define
\bea
\psi(x,u,p,z_{2}) := b_u^*(x,u)p + L_u(x,u) + h_u^{j,*}(x,u)z_{2}^j.
\eea
Then
\beq{nabla-J}\barray
\nabla_u J(u_r^{\iota,Y})
\ad = \EE^\QQ [\psi(x_r, u_{r}^{\iota,Y}|_t, p_r, z_{2,r}) |\CF_t^Y] \\
\aad = \int_{S_h} \EE^\QQ \big[\psi(x_r, u_r^{\iota,Y}|_t, p_r, z_{2,r}) \big| x_t= x\big]\cdot p\big(x|\CF_t^Y \big) dx,
\earray\eeq
where $p(x_t|\CF_t^Y)$ is the probability density function (pdf) of the law $x_t$ given the observation $\CF_t^Y$. This is where nonlinear filtering plays a crucial role in computing partially observed optimal control. The pdf $p(x_t| \CF_t^Y)$ will be approximated by the empirical distributions $\pi(x_t|\CF_t^Y)$ using particles. We will use the Bayesian
framework and branching particle filtering algorithm to update $\pi$.

\subsection{The branching algorithm}
In particle filtering algorithm, a well-known issue is the weight degeneracy problem. That is, after several time steps, the weights of most particles tend to zero, making the particles concentrate on a few samples. It reduces the effective particle size in the algorithm. To overcome this problem, several resampling strategies have been applied, such as systematic resampling, residual resampling, and stratified resampling. However, instead of using resampling strategies, we will use the branching algorithm in this paper because it produces fewer errors and variance; see \cite{WWX24} and \cite{BC09} for comprehensive study on the branching algorithm.

To start, we divide the time interval $[0,T]$ into equal lengths $[k\varrho, (k+1)\varrho)$, where $k>0$ and $\varrho$ is the time step size. At time $0$, the particle system consists of $S$ particles of equal weight $1/S$ distributed at locations $\{x_0^{\vsg,S}\}_{\vsg=1}^S$. The approximating measure at time $0$ is then given by $\pi_0^S :=\frac{1}{S}\sum_{\vsg=1}^S \dl_{x_0^{\vsg,S}}$, where $\dl_x$ is the Dirac measure at $x$.

During the time interval $[k\varrho,(k+1)\varrho)$, the particles all move with the same law as the signal/state $x$. That is for $t\in [k\varrho, (k+1)\varrho)$,
\beq{xt-int}\barray
x_t^{\vsg,S} \ad = x_{k\varrho}^{\vsg, S} + \int_{k\varrho}^t \big(\Dl x_r^{\vsg,S} + b(x_r^{\vsg, S}, u_r) \big) dr -\sum_{j=1}^d \int_{k \varrho}^t g^j(x_r^{\vsg, S}) h^j(x_r^{\vsg, S}, u_r)dr \\
\aad \quad + \sum_{i=1}^N \int_{k\varrho}^t \sg^i(x_r^{\vsg, S}) dW_r^i +\sum_{j=1}^d \int_{k\varrho}^t g^j(x_r^{\vsg, S})dY_r^j ,
\earray\eeq
where $\{W_r^i\}_{i=1}^N$ are mutually independent Brownian motions, independent of $Y$. Remember we will use finite element approximation for spatial discretization, \eqref{xt-int} turns out to be a SDE. The weight $\lbar M_t^{\vsg, S}$ is of the form
\beq{lbarM}\barray\ad
\lbar M_t^{\vsg,  S}=\frac{M_t^{\vsg, S}}{\sum_{l=1}^S M_t^{l, S}},
\earray\eeq
where
\beq{Mt-int}\barray\ad
M_t^{\vsg, S} := \exp\Big(\int_{k\varrho}^t h(x_r^{\vsg, S}, u_r) dY_r - \half \int_{k\varrho}^t |h(x_r^{\vsg, S}, u_r)|^2 dr \Big), \quad t\in [k\varrho, (k+1)\varrho).
\earray\eeq
For $t\in [k\varrho, (k+1)\varrho)$, we define the empirical measure of particles at time $t$ as
\bea\ad
\pi_t^S = \sum_{\vsg=1}^S \lbar M_t^{\vsg, S} \dl_{x_t^{\vsg, S}}.
\eea
At the end of the interval, each particle is branched into a random number $o_{(k+1)\varrho}^{\vsg, S}$ of offsprings. Each offspring particle initially inherits the spatial position of its parent. Then $o_{(k+1)\varrho}^{\vsg, S}$ is $\CF_{(k+1)\varrho}$-adapted, and
\beq{o-tau}
o_{(k+1)\varrho}^{\vsg,S}:= \left\{\barray
\left[S \lbar M_{(k+1)\varrho}^{\vsg, S}\right],\ad \text{ with probability } 1-\left\{S \lbar M_{(k+1)\varrho}^{\vsg, S}\right\},\\

\left[S \lbar M_{(k+1)\varrho}^{\vsg, S}\right]+1, \ad \text{ with probability } \left\{S \lbar M_{(k+1)\varrho}^{\vsg, S}\right\},
\earray\right.
\eeq
where $[x]$ is the largest integer smaller than $x$ and $\{x\}$ is the fractional part of $x$, that is, $\{x\}=x-[x]$, and
\bea\ad
\lbar M_{(k+1)\varrho}^{\vsg, S} := \lbar M_{(k+1)\varrho -}^{\vsg, S} = \lim_{t\uparrow (k+1)\dl} \lbar M_{t}^{\vsg,S}.
\eea
Then the direct computation from \eqref{o-tau} implies
\bea
\EE\Big[o_{(k+1)\varrho}^{\vsg,S} \Big| \CF_{(k+1)\varrho-} \Big]  = S \lbar M_{(k+1)\varrho}^{\vsg, S}.
\eea
Moreover, the minimal of conditional variance of the number of offspring is
\bea\ad
\EE\Big[ \Big(o_{(k+1)\varrho}^{\vsg, S} \Big)^2 \Big| \CF_{(k+1)\varrho-}\Big] - \Big( \EE \Big[o_{(k+1)\varrho}^{\vsg,S} \Big| \CF_{(k+1)\varrho-} \Big] \Big)^2 \\
\aad = \left\{S \lbar M_{(k+1)\varrho}^{\vsg, S} \right\} \Big(1-\left\{S \lbar M_{(k+1)\varrho}^{\vsg, S}\right\} \Big),
\eea
see \cite[Exercise 9.1]{BC09} with its proof in \cite[Proposition 2]{WWX24}. We will control the branching process such that the total number of particles in the system remains at $S$, that is
$\sum_{\vsg=1}^S o_{(k+1)\varrho}^{\vsg,S}=S.$

Consequently, combining the conditional SGD algorithm in \eqref{cond-SGD}, numerical approximation schemes to solve forward-backward SPDEs and BSDEs for \eqref{spde-num}, \eqref{pq}, and \eqref{z} in \cite[Section 3,4]{BCQ25}, as well as the branching particle filtering algorithm above, we obtain a full numerical algorithm to solve partially observed optimal control for SPDEs where the state and observation have finite-dimensional correlated noise. Moreover, we notice that in the computation of the gradient of the cost functional \eqref{nabla-J}, BSPDEs \eqref{pq}, and BSDEs \eqref{z}, one needs to compute conditional expectations. They can be approximated by the Monte-Carlo simulations using samples. In practice, this becomes very costly when the dimension of system state is high. Motivated by the stochastic approximation algorithm, specifically by the stochastic gradient descent algorithm, we utilize a single realization of trajectory to represent the conditional expectation; see details in \cite{BCQ25,ABYZ20} and \cite{KY03}.
We summarize our numerical algorithm in Algorithm \ref{alg:branch}.

\begin{algorithm}
\caption{ 
Finite element based SGD with branching particle filtering}

\label{alg:branch}
\begin{algorithmic}
\State{Initialize initial distribution $\pi_0$, set parameters $T, N, N_T, \varrho, S, n_{\text{SGD}}, n_{\text{FE}}$.}

\For{$\vsg =1,2,\dots, S$}
\State{Sample $x_0^{\vsg,S}$ from $\pi_0$;}

\State{Set $M_0^{\vsg,S}=1$.
}
\EndFor
\While{$n=0,1,\dots, N_T$}

\State{Initialize
$\{u_{t_k}^{0,Y}|_{t_n}\}_{k=n}^{N_T}$, $t_k:= k\varrho$, and set the learning rate $\al$.}

\For{$\iota=1,2,\dots, n_{\text{SGD}}$}

\State (i) \underline{Initialization}: set initial state $x_{t_n}^{\wdh \vsg, S}$ at time $t_n$, where $x_{t_n}^{\wdh \vsg, S}$ is randomly
\Statex \hspace{\algorithmicindent}\hspace{2.6em}
selected from
particle cloud $\{x_{t_n}^{\vsg, S}\}_{\vsg=1}^S$ at time $t_n$.

\State (ii) \underline{FSPDE}: compute one single realization $\{x_{t_k}^{\wdh \vsg, S}\}_{k=n}^{N_T}$ of FSPDE \eqref{xt-int}.

\State (iii) \underline{BSPDE}: compute one single realization $\{p_{t_k}^{\wdh \vsg, S}, q_{1,t_k}^{\wdh \vsg, S}, q_{2,t_k}^{\wdh \vsg, S} \}_{k=N_T}^{n}$ of
\Statex \hspace{\algorithmicindent}\hspace{3.0em} BSPDE \eqref{pq} and one single realization of $\{z_{t_k}^{\wdh \vsg, S}, z_{1,t_k}^{\wdh \vsg, S}, z_{2,t_k}^{\wdh \vsg, S}\}_{k=N_T}^{n}$ of

\Statex \hspace{\algorithmicindent}\hspace{3.0em} BSDE  \eqref{z} corresponding to the trajectory $\{ x_{t_k}^{\wdh \vsg, S}\}_{k=n}^{N_T}$.

\State (iv) \underline{SGD}: iteratively update the control using conditional SGD \eqref{cond-SGD}.
\EndFor

\State{The estimated optimal control at time $t_n$ is given by $\wdh{u}^{Y}_{t_n} = u_{t_n}^{n_{\text{SGD}}, Y}$.}

\State{Generate the location and weights $(x_{(n+1)\varrho}^{\vsg, S}, M_{(n+1)\varrho}^{\vsg, S})$ by \eqref{xt-int} and \eqref{Mt-int}, resp.}
\For{$\vsg =1, 2, \dots, S$}
\State{Compute the normalized weight $\lbar M_{(n+1) \varrho}^{\vsg, S}$ by \eqref{lbarM} letting $t=(n+1)\varrho$.}
\EndFor
\For{$\vsg'= 1, 2, \dots, S$.}

\State
\strut (i) Compute the random number $o_{(n+1)\varrho}^{\vsg', S}$ of offsprings for the $\vsg'$-particle in

\Statex \hspace{\algorithmicindent}\hspace{1.4em}
particle cloud with location/weights $(x_{n \varrho}^{\vsg, S}, M_{(n +1) \varrho}^{\vsg, S})$ by \eqref{o-tau}.

\State{(ii) Replace each particle by $o_{(n+1)\varrho}^{\vsg',S}$ offsprings such that $\sum_{\vsg'=1}^S o_{(n+1) \varrho}^{\vsg', S} = S$.}

\EndFor
\EndWhile
\end{algorithmic}
\end{algorithm}

\section{Numerical experiments}\label{sec:num}
In the following numerical experiments, we use nFE$=400$ finite elements and a time step size $\Dl t= 0.01$ with terminal time $T=1$. The stochastic gradient descent algorithm is run for nSGD $=1000$ iterations. For demonstration, we consider the following partially controlled linear stochastic heat equation. However, our numerical algorithm in \secref{sec:algo} is general and can be applied to nonlinear SPDEs; see \cite[Section 5]{BCQ25} for numerical experiments in nonlinear cases.
\begin{exm}[Stochastic heat equation]\rm{
We consider the state satisfies following stochastic heat equation driven by additive cylindrical Wiener process and by finite many Brownian motions with multiplicative coefficients:
\beq{heat}\left\{\barray
dx_t = [\Dl x_t + u_t] + 0.05dW_t + \sum_{j=1}^d 0.03(x_t+1) e_j dB_t^j \\
x_0  = x_0 \in L^2(0,10)
\earray\right.
\eeq
with $L=10$ and Dirichlet boundary condition $x_0 =0$. The observation process is taken as
$dY_t = h(x_t, u_t) dt + dB_t,\; Y_0=0\in \rr^d, d=5,$
where $h(x)=\arctan(\qv{x+u,\sg_1}_{L^2(\La)}, \dots, \qv{x+u, \sg_d}_{L^2(
\La)})$ for pre-selected elements $\sg_1,\dots$, $\sg_d$ $\in L^2(\La)$. The use of pre-selected elements $\{\sg_1,\dots, \sg_d\}$ is natural in applications. In complex systems such as climate models, turbulence, and water movement in natural environments, observing trajectories at every spatial point is infeasible. Instead, monitoring personnel systematically select some detection points and put sensors to collect state trajectory information. We denote these detection points as $\sg_1,\dots, \sg_d$ that depend only on the spatial variable. Their numerical settings can be found in \cite[Section 5]{BCQ25}. The cost functional we consider is
\bea\ad
J(u)= \EE^\QQ \int_0^1 \half \big[\|x_t\|_{L^2(0,10)}^2 + \|u_t\|^2_{L^2(0,10)} \big] dt + \half \EE^\QQ\|x_T\|_{L^2(0,10)}^2.
\eea
The BSDE \eqref{z} remains the same and the BSPDE in \eqref{pq} becomes
\bea
dp_t \ad = -[\Dl p_t + g_x^j(\wdh x_t) q_{2,t}^j+ h_x^{j,*}(\wdh x_t, \wdh u_t) z_{2,t}^j - g^j(\wdh x_t) h_x^j(\wdh x_t, \wdh u_t) p_t \\
\aad \qquad  + \ell_x(\wdh x_t, \wdh u_t)]dt+ q_{1,t} dW_t + q_{2,t}^j dB^j, \quad q_T= m_x(\wdh x_T).
\eea
Take $\al=0.001$, after $1000$ iterations, we obtain an approximate cost $J \approx 0.8464$. \figref{fig:heat-unc-state} shows a sample path of the stochastic heat equation \eqref{heat} without control, while \figref{fig:heat-control-state} presents the approximated partially observed optimal control (left) and the corresponding sample path of the stochastic heat equation (right).
\begin{figure}[H]
\centering
\includegraphics[width=0.48\textwidth]{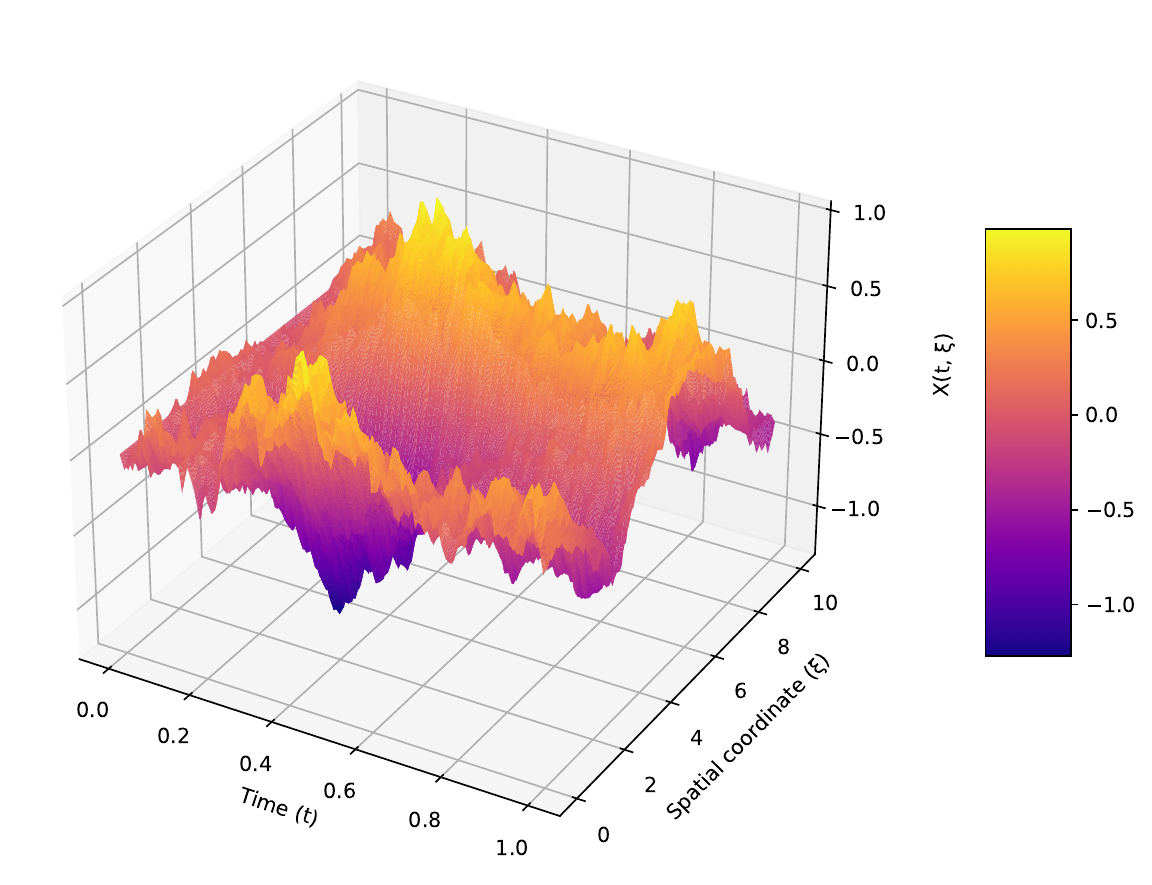}
\label{fig:heat-unc-state}
\caption{Sample path of stochastic heat equations without control.}
\end{figure}

\begin{figure}[H]
\includegraphics[width=0.48\textwidth]{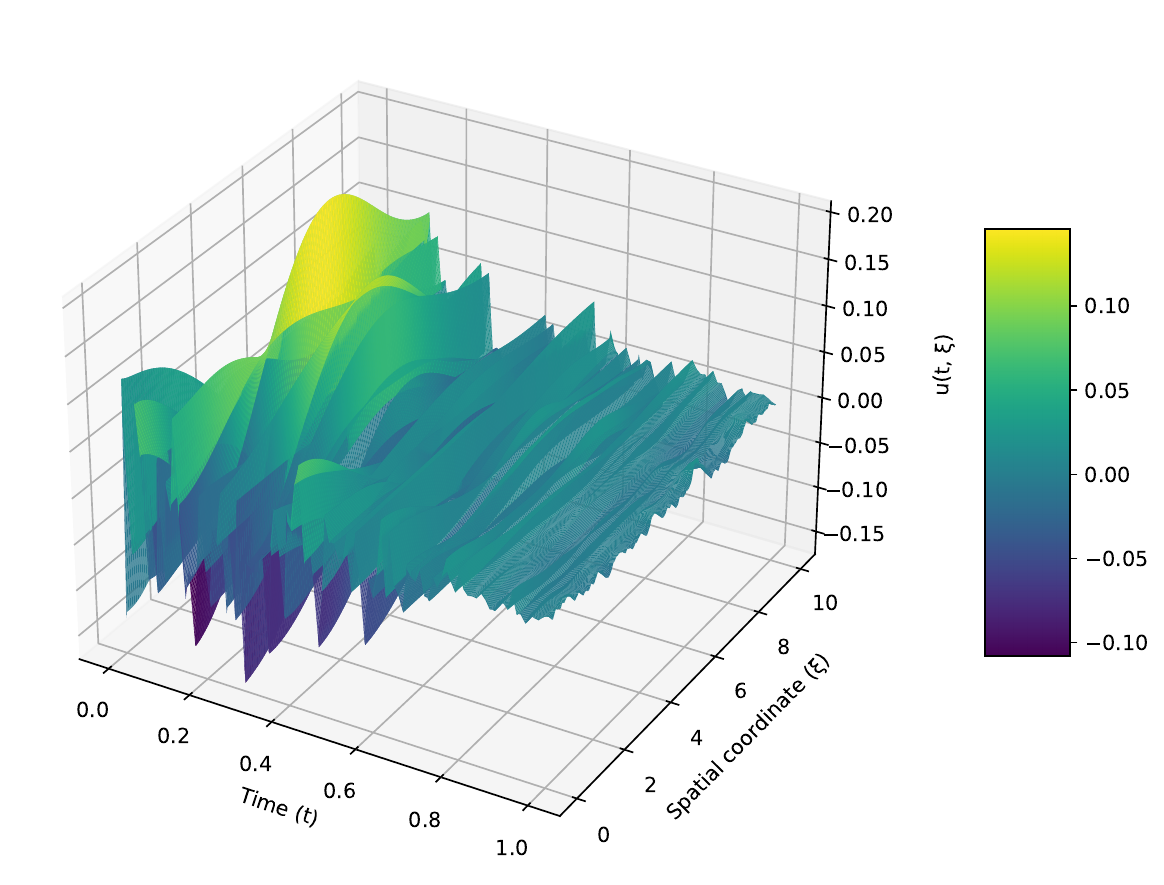}
\includegraphics[width=0.48\textwidth]{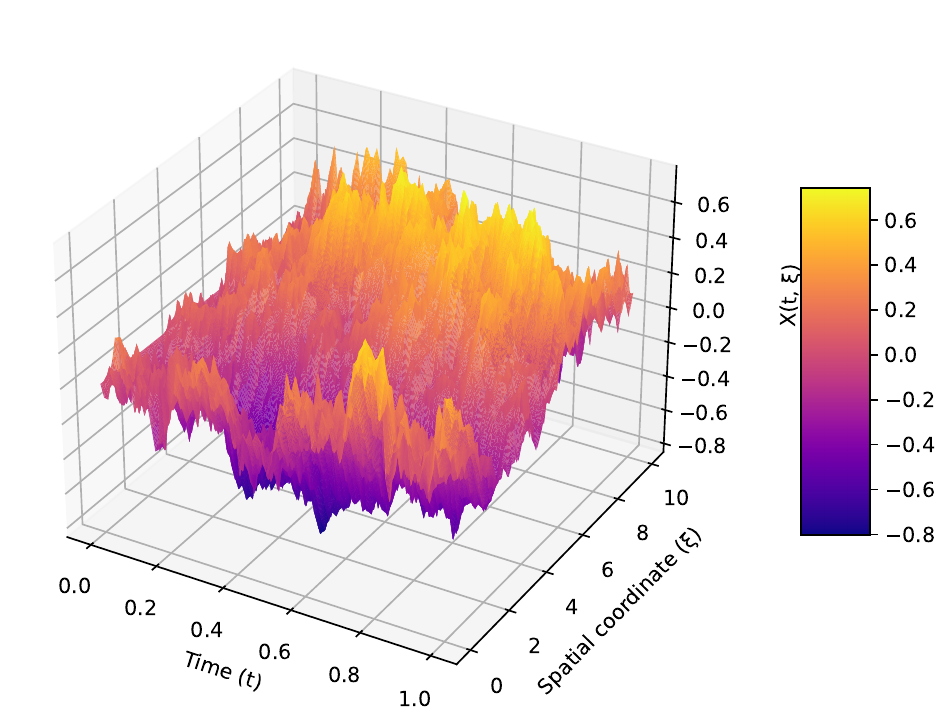}
\caption{Sample path of approximated partially observed optimal control and the corresponding sample path of  optimal controlled state.}
\label{fig:heat-control-state}
\end{figure}
}
\end{exm}

\section{Concluding remarks}
In this paper, we establish Peng's maximum principle for partially observed optimal control of SPDEs with finite-dimensional observations that have correlated noise with the state process. To approximate optimal control, we develop an efficient numerical algorithm that combines finite element approximations for solving forward-backward SPDEs, branching particle filtering for state estimation, and stochastic gradient descent algorithms for the optimization of the control.

The Laplacian $\Dl$ in the state equation \eqref{spde} can be replaced by a general operator $A: \CD(A) \subset L^2(\La) \to L^2(\La)$ of the quadratic form
$\int_\La a(\la) (\partial_\la x)^2(\la)d\la, \; x\in H_0^1(\La),$
for some $a\in L^\infty(\La)$ with $a(\la) \geq a_0 >0$. Moreover, the one-space dimension can be generalized to a higher-space dimension, but one needs a higher space regularity for the noise coefficients $\sg$ and $g$. They are required to guarantee that the operator $\pi$ is a continuous operator; see details in \cite{SW21} and \cite[Section 7.38]{AF03}.

Future research will focus on analyzing the algorithm's convergence, establishing  sufficient conditions for the existence of the partially observed optimal control problem, and exploring applications to more complex real-world models.


\bibliographystyle{siamplain}

\end{document}